\documentclass[a4paper,12pt]{article}
\usepackage{latexsym}
\usepackage{mathrsfs}
\usepackage{amsfonts,amsmath,amssymb}
\usepackage{indentfirst}
\usepackage{subeqnarray}
\usepackage{verbatim}
\usepackage{graphicx}
\usepackage{epstopdf}
\usepackage{algorithm}
\usepackage{algorithmic}
\usepackage{subfigure}
\usepackage{color}
\usepackage{multirow}
\usepackage{float}

\allowdisplaybreaks[4]

\newtheorem{theorem}{Theorem}[section]
\newtheorem{lemma}[theorem]{Lemma}

\newtheorem{definition}[theorem]{Definition}

\newtheorem{remark}[theorem]{Remark}

\numberwithin{equation}{section}
\newenvironment{proof}[1][Proof]{\textbf{#1.} }
{\ \rule{0.75em}{0.75em}\smallskip}

\begin{document}

\begin{center}
\Large\bf SVD method for sparse recovery
\end{center}

\begin{center}
Long Li\footnotemark[1] \quad and \quad Liang Ding\footnotemark[2]$^{,*}$

\footnotetext[1]{Department of Mathematics, Northeast Forestry University, Harbin 150040, China;
e-mail: {\tt 15146259835@nefu.edu.cn}.}
\footnotetext[2]{Department of Mathematics, Northeast Forestry University, Harbin 150040, China;
e-mail: {\tt dl@nefu.edu.cn}. The work of this author was supported by the Fundamental Research Funds for the Central Universities (no.\ 2572021DJ03).} 

\renewcommand{\thefootnote}{\fnsymbol{footnote}}
\footnotetext[1]{Corresponding author.}
\renewcommand{\thefootnote}{\arabic{footnote}}
\end{center}

\medskip
\begin{quote}
{\bf Abstract.} Sparsity regularization has garnered significant interest across multiple disciplines, including statistics, imaging, and signal processing. Standard techniques for addressing sparsity regularization include iterative soft thresholding algorithms and their accelerated variants. However, these algorithms rely on Landweber iteration, which can be computationally intensive. Therefore, there is a pressing need to develop a more efficient algorithm for sparsity regularization. The Singular Value Decomposition (SVD) method serves as a regularization strategy that does not require Landweber iterations; however, it is confined to classical quadratic regularization.  This paper introduces two inversion schemes tailored for situations where the operator $K$ is diagonal within a specific orthogonal basis, focusing on $\ell_{p}$ regularization when $p=1$ and $p=1/2$. Furthermore, we demonstrate that for a general linear compact operator $K$, the SVD method serves as an effective regularization strategy. To assess the efficacy of the proposed methodologies, We conduct several numerical experiments to evaluate the proposed method's effectiveness. The results indicate that our algorithms not only operate faster but also achieve a higher success rate than traditional iterative methods.
\end{quote}



\section{Introduction}

\par This paper addresses the challenge of solving an ill-posed operator equation of the form 
\begin{equation}\label{equ1.1}
Kx=y,
\end{equation}
where $K:X\rightarrow Y$ is a linear and compact operator, $x$ is sparse vector, $X$ is a Hilbert space, and $Y$ is a Hilbert space equipped with norm $\|\cdot\|_Y$. Throughout this paper, $\langle \cdot,\cdot\rangle$ denotes the inner product in the $X$ and $Y$ spaces. In practical applications, the data $y$ is not available in its exact form; instead, only an approximation $y^{\delta}$ is known, satisfying $\|y^{\delta}-y\|_Y\leq \delta$ for a small $\delta>0$. 

\par The most prevalent approach to solving the ill-posed operator equation (\ref{equ1.1}) is through a sparsity regularization with $\ell^{p}$ penalties established in \cite{DDD04}, formulated as
\begin{equation}\label{equ1.2}
    \min\limits_{x\in X}\left\{\|Kx-y^{\delta}\|_Y^{2}+\sum_{n\in\mathbb{N}} w_{n}\left|\left<x,\phi_{n}\right>\right|^{p}\right\},
\end{equation}
where $w_{n}=\alpha>0$ for all $n\in\mathbb{N}$, $\left(\phi_{n}\right)_{n\in\mathbb{N}}$ gives an orthonormal basis of $X$, and $\left\|Kx-y^{\delta}\right\|_Y^{2}$ serves as the fidelity term that measures the discrepancy between the estimated data $y^{\delta}$ and the model output $Kx$. Over the past two decades, $\ell_1$ sparsity regularization has gained significant traction, prompting extensive research into issues of well-posedness and the development of algorithms for sparse recovery problems, as noted in \cite{DDD04,F2010,JM12} and the references therein. Additionally, several non-convex sparsity regularization methods, such as $\ell_p$ (for $0\le p<1$) \cite{BL09,G10} and $\alpha\ell_1-\beta\ell_2$ \cite{DH19,DH20} regularization, can yield sparser solutions. This paper concentrates on a regularizers for a sparsity regularization with $\ell_{p}$ penalties, which is a key emphasis on developing numerical algorithms for sparsity regularization.

\subsection{Some relative works}

\par A widely utilized algorithm, known as the Iterative Soft Thresholding Algorithm (ISTA), was introduced in \cite{DDD04} for addressing \eqref{equ1.2}. Subsequently, ISTA was reinterpreted within a convex analysis framework as a proximal algorithm \cite{CW05} and also as a generalized conditional gradient algorithm \cite{BLM09}. ISTA employs a gradient descent method that updates all coordinates of the solution simultaneously at each iteration. This approach can lead to arbitrary slowness and is computationally intensive. The search for accelerated variations of ISTA has become a popular research area, resulting in the proposal of several faster algorithms. A comparison of various accelerated algorithms, including the \verb+"+Fast ISTA\verb+"+ (\cite{BT09}), is presented in \cite{LBDZZ09}. By applying a smoothing technique from Nesterov (\cite{N05}), a swift and accurate first-order method was proposed in \cite{BBC11} to tackle large-scale compressed sensing challenges. Furthermore, \cite{DC15} introduced a simple heuristic adaptive restart technique, which can significantly enhance the convergence rate of accelerated gradient schemes. The convergence of the iterates of the \verb+"+Fast Iterative Shrinkage/Thresholding Algorithm\verb+"+ is established in \cite{CD15}. In \cite{OBGXY05}, a novel iterative regularization procedure for inverse problems, leveraging Bregman distances, is explored, with numerical results demonstrating significant improvements over conventional methods. An explicit algorithm based on a primal-dual approach for minimizing an $\ell_1$-penalized least-squares function with a non-separable $\ell_1$ term is proposed in \cite{LV11}. Additionally, \cite{FPRW16} investigates an iteratively reweighted least squares algorithm along with corresponding convergence analysis for the regularization of linear inverse problems with sparsity constraints. Several accelerated projected gradient (PG) methods are documented in \cite{BF08,DFL08,FNW07,WNF09}. However, all algorithms mentioned in the previous context are either ST or PG Landweber iterative algorithms. Essentially, these represent the combination of ST/PG with the Landweber iteration, expressed as follows
\begin{equation}\label{equ1.3}
x^{n+1}={\rm prox}\left(x^n-K^{*}\left(Kx^n-y^{\delta}\right)\right).
\end{equation}
This is a gradient descent algorithm known for its typically slow convergence. Furthermore, some experiments in sparse recovery indicate that the probability of successful recovery using these iterative algorithms diminishes as the support of the  $\ell_1$-minimum norm solution increases \cite{LY18,YLHX15}.

A convenient admissible regularization strategy to bypass the Landweber iteration \eqref{equ1.3} is provided by the singular value decomposition (SVD) method \cite{EHN1996,K1996}. Let $K:X\rightarrow Y$ be a linear compact operator, and denote by $\{(\sigma_{n},v_{n},u_{n})\}_{n\in\mathbb{N}}$ a singular system for $K$. The operator $\mathcal{R}_{\alpha}:Y\rightarrow X$, where $\alpha>0$, is defined by the equation
\begin{equation}\label{equ1.4}
    \mathcal{R}_{\alpha}y:=\sum_{n\in\mathbb{N}}\frac{q\left(\alpha, \sigma_n\right)}{\sigma_{n}}\left<y,u_{n}\right>v_{n}\ ,\ y\in Y.
\end{equation}
This operator serves as a regularization operator with the property that $\left\|\mathcal{R}_{\alpha}\right\|\leq c(\alpha)$, where the function $q$ is referred to as a regularizing filter for $K$. Specifically, if $q(\alpha, \sigma)=\sigma^2/(\alpha+\sigma^2)$, then $\mathcal{R}_{\alpha}$ corresponds to the Tikhonov regularization method. If $q(\alpha, \sigma)=1-(1-\alpha\sigma^2)^{1/\alpha}$, $\mathcal{R}_{\alpha}$ represents the Landweber regularization method. In the case where
\begin{align*}
    q(\alpha, \sigma)=
    \left\{             
             \begin{array}{ll}
             1, &  \sigma^2\ge\alpha, \\
             0, &  \sigma^2<\alpha, 
             \end{array}
    \right.
\end{align*}
$\mathcal{R}_{\alpha}$ signifies a truncated singular value regularization method. As a direct inversion technique, the SVD method is particularly efficient and easy to implement. However, since the equation \eqref{equ1.4} constitutes a linear regularization method, the operator $\mathcal{R}_{\alpha}$ is regarded as a family of bounded linear operators that converge pointwise to $K^{\dag}$ as $\alpha\rightarrow 0$. Consequently, the SVD method is confined to classical quadratic regularization, which presumes that the true solution exhibits a sufficient degree of smoothness. In contrast, there is a growing interest in nonlinear regularization methods involving nonlinear mappings $\mathcal{R}_{\alpha}$ (which may even be multivalued) \cite{BB18}. There exists significant potential for applying the SVD method to other penalty terms, however, to the best of our knowledge, there is little work available in the literature regarding this approach.

\subsection{Contribution and organization}

\par In this paper, we substitute the orthogonal vector $\left(\phi_{n}\right)_{n\in\mathbb{N}}$ with the right singular vector $\left(v_{n}\right)_{n\in\mathbb{N}}$ of $K$. When $p=1$ and $\omega_{n}=\alpha$ for $n=1,2,\cdots$, the proposed minimization approach encourages sparsity in the expansion of $x$ concerning the $v_{n}$, which allows us to rewrite the sparsity regularization problem \eqref{equ1.2} as follows
\begin{equation}\label{equ1.5}
    \min\limits_{x\in X}\left\{\Phi_{\alpha}(x)=\left\|Kx-y^{\delta}\right\|_Y^{2}+\alpha \sum_{n\in\mathbb{N}}\left|\left<x, v_{n}\right>\right|^{p}\right\}.
\end{equation}
\par This paper aims to improve computational efficiency and enhance the success rate of recovery in solving inverse sparse problems. It presents four key contributions:

\textbullet\ When $p=1$, we introduce a nonlinear regularization operator ($\ell^{1}$-SVD) of the form 
\begin{equation}\label{equ1.6}
      \mathcal{R}_{\alpha}y^{\delta}:=\sum_{n\in\mathbb{N}}\frac{1}{\sigma_{n}^{2}}\mathcal{S}_{\alpha}\left(\sigma_{n}\left< y^{\delta},u_{n}\right>\right)v_{n}
\end{equation}
inspired by the SVD method for the sparsity regularization \eqref{equ1.5} when $p=1$ and $K$ happens to be diagonal under a set of orthogonal bases
\begin{equation}\label{equ1.7}
    \min\limits_{x\in X}\left\{\Phi_{\alpha}(x)=\left\|Kx-y^{\delta}\right\|_{Y}^{2}+ \alpha\sum_{n\in\mathbb{N}}\left|\left<x,v_{n}\right>\right|\right\},
\end{equation}
where $\mathcal{S}_{\alpha}$ denotes a soft thresholding operator governed by a regularization parameter $\alpha>0$. 

\textbullet\ When $p=1/2$, we also propose a nonlinear regularization operator ($\ell^{1/2}$-SVD) of the form 
\begin{equation}\label{equ1.8}
    \mathcal{R}_{\alpha}y^{\delta}:=\sum_{n\in\mathbb{N}}\mathcal{H}_{\alpha}\left(\sigma_{n}^{1/3}\left<y^{\delta}, u_{n}\right>\right)v_{n}
\end{equation}
inspired by the SVD method for the sparsity regularization \eqref{equ1.5} when $p=1/2$ and $K$ happens to be diagonal under a set of orthogonal bases
\begin{equation}\label{equ1.9}
    \min\limits_{x\in X}\left\{\Phi_{\alpha}(x)=\left\|Kx-y^{\delta}\right\|_{Y}^{2} +\alpha\sum_{n\in\mathbb{N}}\left|\left<x,v_{n}\right>\right|^{1/2}\right\},
\end{equation}
where $\mathcal{H}_{\alpha}$ is a half thresholding operator introduced in \cite{XCXZ12}, with a regularization parameter $\alpha>0$. 

\textbullet\ Our analysis reveals that when $K$ is diagonal under a set of orthogonal bases, the solution obtained from the $\ell^{1}$-SVD algorithm serves as a minimizer for \eqref{equ1.7}, while the solution derived from the $\ell^{1/2}$-SVD algorithm acts as a stationary point for \eqref{equ1.9}. Additionally, we demonstrate that for a general linear compact operator $K$, the SVD method employing nonlinear soft thresholding and half thresholding functions is an effective regularization strategy. We further establish the regularization property and provide an error estimate for the case when $p=1$.

\textbullet\ Experimental results consistently highlight the advantage of our proposed method over several conventional iterative algorithms, particularly concerning algorithm execution time and the probability of successful recovery. Unlike traditional iterative algorithms, which require the calculation of \eqref{equ1.3} in every iteration, the SVD algorithm necessitates only a single computation of singular value decomposition. Once a singular value system is established, a regularized solution can be determined using \eqref{equ1.6} or \eqref{equ1.8}. Thus, the SVD algorithm is significantly faster compared to iterative methods.

\par This paper is structured as follows. In the next section, we will introduce the notation and review the results of the SVD method. In Section \ref{sec3}, we will present a novel approach that integrates sparsity regularization with $\ell^{1}$ and $\ell^{1/2}$ penalties within the framework of SVD in scenarios where the operator $K$ is diagonal within a specific orthogonal basis. Furthermore, we demonstrate that for a general linear compact operator $K$, the SVD method that utilize nonlinear soft thresholding and half thresholding functions serves as an effective regularization strategy. Finally, we will share results from numerical experiments focused on compressive sensing and image deblurring in Section \ref{sec4}.

\section{Preliminaries}\label{sec2}

Before we delve into the proposed algorithms, let us first introduce some notations and key results about Tikhonov regularization and the SVD method. For further details, refer to \cite{DDD04,EHN1996,K1996}. 

\begin{definition}\label{def2.1}
    For $\alpha>0$, a minimizer of the regularization function $\mathcal{J}_{\alpha}(x)$ is denoted by
    \begin{equation*}
    x_{\alpha}^{\delta}\in {\rm arg}\min_{x\in X}\left\{\mathcal{J}_{\alpha}(x)= \|Kx-y^{\delta}\|_{Y}^{2}+\alpha\sum_{n\in\mathbb{N}}\left|\left<x,v_{n}\right>\right|^{p}\right\}.
    \end{equation*}
\end{definition}

\begin{definition}\label{def2.2}
    An element $x\in X$ is called sparse if ${\rm supp}(x):=\left\{i\in\mathbb{N} \mid x_{i}\neq 0\right\}$ is finite, where $x_{i}$ represents the $i${\rm th} component of $x$. The quantity $\|x\|_{\ell_{0}}:={\rm supp}(x)$ denotes the cardinality of ${\rm supp}(x)$. If $\|x\|_{\ell_{0}}=s$ for some $s\in\mathbb{N}$, then $x\in X$ referred to as $s$-sparse.
\end{definition}

\begin{definition}\label{def2.3}
    Let $K:X\rightarrow Y$ be a linear compact operator between Hilbert spaces. Then there exists a unique linear compact operator $K^{*}:Y\rightarrow X$ such that
    \begin{equation*}
        \left<Kx,y\right>=\left<x,K^{*}y\right>
    \end{equation*}
    for all $x\in X$ and $y\in Y$. This operator $K^{*}$ is referred to as the adjoint operator of $K$. If $X=Y$, the operator $K$ is called self-adjoint if $K^{*}=K$.
\end{definition}

\begin{definition}\label{def2.4}
    Let $X$ and $Y$ be Hilbert spaces, and let $K:X\rightarrow Y$ be a linear compact operator with an adjoint operator denoted by $K^{*}:Y\rightarrow X$. The square roots $\sigma_{n}=\sqrt{\lambda_{n}}$, where $n\in\mathbb{N}$, of the eigenvalues $\lambda_{n}$ of the self-adjoint operator $K^{*}K:X\rightarrow X$ are referred to as the singular values of $K$.
\end{definition}

\par It is important to note that the eigenvalues of $K^{*}K$ are positive. The equation $K^{*}Kx=\lambda x$ implies that $\lambda\left<x,x\right>=\left<K^{*}Kx,x\right>=\left<Kx,Kx\right>=\|Kx\|^{2}> 0$ with $\lambda>0$, which is a result of the definition of the adjoint operator.

\begin{lemma}\label{lem2.5}
    Let $K:X\rightarrow Y$ be a linear compact operator, with $K^{*}:Y\rightarrow X$ denoting its adjoint operator. The collection of singular values $\{\sigma_{n}\}_{n\in\mathbb{N}}$ represents an ordered sequence of positive singular values for $K$, arranged in decreasing order and taking into account their multiplicities. The $\{v_{n}\}_{n\in\mathbb{N}}$ and $\{u_{n}\}_{n\in\mathbb{N}}$ are corresponding complete orthonormal system of eigenvectors of $K^{*}K$ (which spans $\overline{\mathcal{R}(K^{*})}=\overline{\mathcal{R}(K^{*}K)}$) and $KK^{*}$ (which spans $\overline{\mathcal{R}(K)}=\overline{\mathcal{R}(KK^{*})}$), respectively. They satisfy the following properties
    \begin{equation}\label{equ2.1}
        Kv_{n}=\sigma u_{n}\ and\ K^{*}u_{n}=\sigma v_{n},\ n\in\mathbb{N}.
    \end{equation}
    The collection $\{(\sigma_{n},v_{n},u_{n})\}_{n\in\mathbb{N}}$ is termed the singular system for $K$. Every element $x\in X$ and $y\in Y$ can be expressed through singular value decomposition as follows
    \begin{equation*}
        x=\sum_{n\in\mathbb{N}}\left<x,v_{n}\right>v_{n}\ and\ y=\sum_{n\in\mathbb{N}}\left<y,u_{n}\right>u_{n}.
    \end{equation*}
    Furthermore, the following equalities hold
    \begin{equation}\label{equ2.2}
        Kx=\sum_{n\in\mathbb{N}}\sigma_{n}\left<x,v_{n}\right>u_{n}\ and\ K^{*}y=\sum_{n\in\mathbb{N}}\sigma_{n}\left<y,u_{n}\right>v_{n},    
    \end{equation}
    where these infinite series converge in the Hilbert space norms of $X$ and $Y$,  respectively. 
\end{lemma}

\par The equation \eqref{equ2.2} is referred to as the singular value expansion and serves as the infinite-dimensional analogue of the well-known singular value decomposition of a matrix.

\begin{lemma}\label{lem2.6}
     If $\{(\sigma_{n},v_{n},u_{n})\}_{n\in\mathbb{N}}$ constitutes a singular system for $K$, then the following holds
    \begin{equation*}
        \left<y,u_{n}\right>=\sigma_{n}\left<x,v_{n}\right>,
    \end{equation*}
    where $x\in X$ and $y\in Y$ for all $n\in\mathbb{N}$.
\end{lemma}

\begin{theorem}\label{the2.7}
    Let $K:X\rightarrow Y$ be a linear compact operator with a singular system $\left\{\left(\sigma_{n},v_{n},u_{n}\right)\right\}_{n\in\mathbb{N}}$. The equation 
    \begin{equation}\label{equ2.3}
        Kx=y
    \end{equation}
    is solvable if and only if 
    \begin{equation*}
        y\in\mathcal{N}(K^{*})^{\bot}\quad and \quad \sum_{n\in\mathbb{N}}\frac{1}{\sigma_{n}^{2}}|\left<y,u_{n}\right>|^{2}<\infty.
    \end{equation*}
    In this case, the solution to \eqref{equ2.3} is given by
    \begin{equation}\label{equ2.4}
        x=\sum_{n\in\mathbb{N}}\frac{\left<y,u_{n}\right>}{\sigma_{n}}v_{n}.
    \end{equation}
\end{theorem}

\section{SVD method for sparsity regularization}\label{sec3}

\par The equation \eqref{equ2.4} illustrates the ill-posedness of the compact operator equation $Kx=y$. When a perturbation $y^{\delta}$ is introduced to $y$, \eqref{equ2.4} is not stable because $\sigma_{n}\rightarrow 0$ $\left(n\rightarrow +\infty\right)$. In general regularization theory, the construction of the regularization operator $\mathcal{R}_{\alpha}$ fundamentally aims to devise a method for filtering out the small singular values $\sigma_{n}$ of the operator $K$. The subsequent results delineate the feasibility of this general construction.

\begin{theorem}\label{the3.1}
    Let $\left\{\left(\sigma_{n}>0, u_{n}, v_{n}\right)\right\}_{n\in\mathbf{N}}$ be a singular system of the compact linear operator $K$. Define the function
    \begin{equation*}
        q\left(\alpha,\sigma_{n}\right): \left(0,\infty\right)\times(0,\left\|K\right\|]\rightarrow \mathbf{R}
    \end{equation*}
    which satisfies the following properties:
    \begin{enumerate}
    \item For all $\alpha>0$ and $0<\sigma_{n}<\left\|K\right\|$, it holds that $\left|q\left(\alpha,\sigma_{n}\right)\right|\leq 1$;
    \item There exists a function $c\left(\alpha\right)$ such that for all $0<\sigma_ {n}<\left\|K\right\|$, the inequality 
        \begin{equation*}
            \left|q\left(\alpha,\sigma_{n}\right)\right|\leq c\left(\alpha\right)\sigma_{n}
        \end{equation*} 
        holds;
    \item For every $0<\sigma_{n}<\left\|K\right\|$, it holds that $\lim_{\alpha\rightarrow 0}q\left(\alpha,\sigma_{n}\right)=1$.
    \end{enumerate}
    Then the operator $\mathcal{R}_{\alpha}:Y\rightarrow X$, $\alpha>0$:
    \begin{equation}\label{equ3.1}
        \mathcal{R}_{\alpha}y=\sum_{n\in\mathbf{N}}\frac{q\left(\alpha,\sigma_{n}\right)} {\sigma_{n}}\left<y,u_{n}\right>v_{n}\ ,\ y\in Y
    \end{equation}
    constitutes a regularization operator with $\left\|\mathcal{R}_{\alpha}\right\|\leq c\left(\alpha\right)$.
\end{theorem}

\par The proof associated with the theorem above can be referenced in \cite[Theorem 2.6]{K1996}. However, as it currently stands, $\mathcal{R}_{\alpha}$ is limited to linear forms for quadratic regularization; thus, we are particularly interested in examining the regularization operator under the premise that $\mathcal{R}_{\alpha}$ is a nonlinear operator, with the intention of further exploration in this domain.

\subsection{SVD method for $\ell^{1}$-SVD regularization operator}\label{sec3.1}

\par In this section, we explore the interplay between the SVD method and sparsity regularization with $\ell^{1}$ penalties, formally defined by the optimization problem
\begin{equation}\label{equ3.2}
    \min_{x\in X}\left\{\Phi_{\alpha}(x)=\left\|Kx-y^{\delta}\right\|_{Y}^{2}+ \alpha\sum_{n\in\mathbb{N}}\left|\left<x,v_{n}\right>\right|\right\},
\end{equation}
where $\alpha>0$ serves as the regularization parameter. In scenarios where the operator $K$ is diagonal within a specific orthogonal basis, we present a direct formulation of the SVD calculation pertinent to the corresponding regularization solution, accompanied by a detailed proof. Under this condition, we adapt the SVD methodology to effectively demonstrate its capacity to minimize the functional incorporation of the non-smooth even non-convex regularization term.

\begin{definition}\label{def3.2}
    For $t\in\mathbb{R}$ and $\alpha>0$, the function $\mathcal{S}_{\alpha}$ is defined as follows
\begin{align}\label{equ3.3}
    \mathcal{S}_{\alpha}(t)=
    \left\{             
             \begin{array}{ll}
             t+\frac{\alpha}{2}, & if \ t\leq -\frac{\alpha}{2}, \\
             0, & if \ |t|\leq \frac{\alpha}{2}, \\
             t-\frac{\alpha}{2}, & if \ t\geq \frac{\alpha}{2}. 
             \end{array}
    \right.
\end{align}
Thus, $\mathcal{S}_{\alpha}$ is referred to as a soft thresholding function mapping $\mathbb{R}$ to $\mathbb{R}$.
\end{definition}

\begin{theorem}\label{the3.3}
    Let $K: X\rightarrow Y$ is a linear compact operator and let $\{(\sigma_{n},u_{n},v_{n})\}_{n\in\mathbb{N}}$ denote a singular system of $K$. If the operator $K$ happens to be diagonal in the $v_{n}$-basis and a minimum solution $x_{\alpha}^{\delta}$ exists for the functional $\Phi_{\alpha}\left(x\right)$ defined in \eqref{equ3.2}, it can be expressed as follows
    \begin{equation}\label{equ3.4}
        x_{\alpha}^{\delta}=\sum_{n\in\mathbb{N}}\mathcal{S}_{\frac{\alpha}{\sigma_{n}^{2}}}\left(\frac{\left< y^{\delta},u_{n}\right>}{\sigma_{n}}\right)v_{n}
        =\sum_{n\in\mathbb{N}}\frac{1}{\sigma_{n}^{2}}\mathcal{S}_{\alpha}(\sigma_{n}\left< y^{\delta},u_{n}\right>)v_{n},
    \end{equation}
    where $\mathcal{S}_{\alpha}$ is the thresholding function defined in Definition \ref{def3.2}. Furthermore, if $x_{\alpha}^{\delta}$ be represented in the form given by \eqref{equ3.4}, it constitutes a minimum solution to the functional  $\Phi_{\alpha}\left(x\right)$ in \eqref{equ3.2} for a regularization parameter $\alpha >0$.    
\end{theorem}

\begin{proof}
    By Lemma \ref{lem2.5} and the operator $K$ is diagonal in the $v_{n}$-basis, the function $\Phi_{\alpha}(x)$ from equation \eqref{equ3.22} can be expressed as follows
    \begin{align}\label{equ3.5}
        \Phi_{\alpha}(x)
        & =\|Kx-y^{\delta}\|_{Y}^{2}+\sum_{n\in\mathbb{N}}\alpha|\left<x,v_{n}\right>| \nonumber\\
        & =\sum_{n\in\mathbb{N}}\left[|\sigma_{n}x_{n}-y^{\delta}_{n}|^{2}+\alpha|x_{n}|\right],
    \end{align}
    where $x_{n}$ represents $\left<x,v_{n}\right>$ and $y^{\delta}_{n}$ indicates $\left<y^{\delta},u_{n}\right>$ for the singular system $\{(\sigma_{n},v_{n},u_{n})\}_{n\in\mathbb{N}}$ of $K$. If $x_{\alpha}^{\delta}$ is a minimum solution of \eqref{equ3.2}, it must satisfy the first-order optimality condition
    \begin{equation}\label{equ3.6}
        2\sigma_{n}\left(\sigma_{n}\left(x_{\alpha}^{\delta}\right)_{n}-y^{\delta}_{n}\right)+\alpha{\rm{sign}}\left(\left(x_{\alpha}^{\delta}\right)_{n}\right)=0.
    \end{equation}
    For the case where $\left(x_{\alpha}^{\delta}\right)_{n}>0$, this leads to the expression $\left(x_{\alpha}^{\delta}\right)_{n}=(1/\sigma_{n}^{2})(\sigma_{n}y^{\delta}_{n}-\alpha/2)$, valid under the condition that $\sigma_{n}y^{\delta}_{n} >\alpha/2$. Conversely, if $\left(x_{\alpha}^{\delta}\right)_{n}<0$, we find $\left(x_{\alpha}^{\delta}\right)_{n}=(1/\sigma_{n}^{2})(\sigma_{n}y^{\delta}_{n}+\alpha/2)$, applicable only when $\sigma_{n}y^{\delta}_{n} <-\alpha/2$. For the scenario where $|\sigma_{n}y^{\delta}_{n}| <\alpha/2$, we set $\left(x_{\alpha}^{\delta}\right)_{n}=0$. In summary, we have
    \begin{equation*}
        \left(x_{\alpha}^{\delta}\right)_{n}=\frac{1}{\sigma_{n}^{2}}\mathcal{S}_{\alpha}(\sigma_{n}\left<y^{\delta}_{n}, u_{n}\right>) 
    \end{equation*}
    and
    \begin{equation*}
        x_{\alpha}^{\delta}=\sum_{n\in\mathbb{N}}\frac{1}{\sigma_{n}^{2}}\mathcal{S}_{\alpha}(\sigma_{n}\left<y^{\delta},u_{n}\right>)v_{n}, 
    \end{equation*}
    where the function $\mathcal{S}_{\alpha}$ is defined in equation \eqref{equ3.3}. 
    
    \par Next, we demonstrate that if \eqref{equ3.4} holds, then $x_{\alpha}^{\delta}$ is a minimum solution for the functional $\Phi_{\alpha}\left(x\right)$ defined in \eqref{equ3.2}. For any $h\in X$, we define $h_{n}=\left<h, v_{n}\right>$. Consequently, we have
    \begin{align}\label{equ3.7}
        \Phi_{\alpha}\left(x_{\alpha}^{\delta}+h\right)
        & =\left\|Kx_{\alpha}^{\delta}+Kh-y^{\delta}\right\|_Y^{2}+ \sum_{n\in\mathbb{N}}\alpha\left|\left(x_{\alpha}^{\delta}\right)_{n}+h_{n}\right| \nonumber\\
        & =\left\|Kx_{\alpha}^{\delta}-y^{\delta}\right\|_Y^{2}+ 2\left<Kx_{\alpha}^{\delta}-y^{\delta},Kh\right>+\left\|Kh\right\|_Y^{2}+ \sum_{n\in\mathbb{N}}\alpha\left|\left(x_{\alpha}^{\delta}\right)_{n}+h_{n}\right| \nonumber\\
        & =\Phi_{\alpha}\left(x_{\alpha}^{\delta}\right)+\left\|Kh\right\|_Y^{2}+ \sum_{n\in\mathbb{N}}\alpha\left|(\left(x_{\alpha}^{\delta}\right)_{n}+h_{n}\right| -\sum_{n\in\mathbb{N}}\alpha\left|\left(x_{\alpha}^{\delta}\right)_{n}\right|+ 2\left<Kx_{\alpha}^{\delta}-y^{\delta},Kh\right>. 
    \end{align}
    Since $x_{\alpha}^{\delta}$ can be expressed as in \eqref{equ3.4}, it satisfies the first-order optimality condition given in \eqref{equ3.6}. Therefore, we have
    \begin{equation*}
        2\left<Kx_{\alpha}^{\delta}-y^{\delta},Kh\right>= \sum_{n\in\mathbb{N}}2\sigma_{n}\left(\sigma_{n}\left(x_{\alpha}^{\delta}\right)_{n} -y^{\delta}_{n}\right) h_{n} =\sum_{n\in\mathbb{N}}-\alpha{\rm{sign}} \left(\left(x_{\alpha}^{\delta}\right)_{n}\right)h_{n}.
    \end{equation*}
    Let us define $\Lambda:=\left\{n\in\mathbb{N} \mid \left(x_{\alpha}^{\delta}\right)_{n}\neq 0\right\}$. Then \eqref{equ3.7} can be recast as
    \begin{align*}
        \Phi_{\alpha}\left(x_{\alpha}^{\delta}+h\right)-\Phi\left(x_{\alpha}^{\delta}\right)
        & =\left\|Kh\right\|_Y^{2}+\sum_{n\notin\Lambda}\left[\alpha|h_{n}|- 2h_{n}\left<K^{*}y^{\delta},v_{n}\right>\right] \nonumber\\
        & \quad +\sum_{n\in\Lambda}\left\{\alpha\left|\left(x_{\alpha}^{\delta}\right)_{n}+h_{n}\right|- \alpha\left|\left(x_{\alpha}^{\delta}\right)_{n}\right| -\alpha{\rm{sign}}\left(\left(x_{\alpha}^{\delta}\right)_{n}\right)h_{n}\right\}.
    \end{align*}
    For $n\notin\Lambda$, since $\left(x_{\alpha}^{\delta}\right)_{n}=0$, it follows from \eqref{equ3.3} that $|\left<K^{*}y^{\delta},v_{n}\right>|\leq\alpha/2$. Thus, we can assert
    \begin{equation*}
        \alpha|h_{n}|-2h_{n}\left<K^{*}y^{\delta},v_{n}\right>\geq 0.
    \end{equation*}
    Now consider the case where $n\in\Lambda$ meaning $\left(x_{\alpha}^{\delta}\right)_{n}\neq 0$. If $\left(x_{\alpha}^{\delta}\right)_{n}>0$, then
    \begin{equation*}
        \alpha\left|\left(x_{\alpha}^{\delta}\right)_{n}+h_{n}\right|- \alpha\left|\left(x_{\alpha}^{\delta}\right)_{n}\right|+h_{n}\left[-\alpha{\rm sign}\left(\left(x_{\alpha}^{\delta}\right)_{n}\right)\right] =\alpha\left[\left|\left(x_{\alpha}^{\delta}\right)_{n}+h_{n}\right|- \left(\left(x_{\alpha}^{\delta}\right)_{n}+h_{n}\right)\right]\geq 0.
    \end{equation*}
    Conversely, if $\left(x_{\alpha}^{\delta}\right)_{n}<0$, then
    \begin{equation*}
        \alpha\left|\left(x_{\alpha}^{\delta}\right)_{n}+h_{n}\right|- \alpha\left|\left(x_{\alpha}^{\delta}\right)_{n}\right|+h_{n}\left[-\alpha{\rm sign}\left(\left(x_{\alpha}^{\delta}\right)_{n}\right)\right] =\alpha\left[\left|\left(x_{\alpha}^{\delta}\right)_{n}+h_{n}\right|+ \left(\left(x_{\alpha}^{\delta}\right)_{n}+h_{n}\right)\right]\geq 0.
    \end{equation*}
    Thus, we can easily conclude that
    \begin{equation*}
        \Phi\left(x_{\alpha}^{\delta}+h\right)-\Phi\left(x_{\alpha}^{\delta}\right) \geq\left\|Kh\right\|_Y^{2}.
    \end{equation*}
    As a result, $x_{\alpha}^{\delta}$ given in \eqref{equ3.4} is indeed a minimum solution for \eqref{equ3.5}. 
\end{proof}

\begin{remark}\label{rem3.3}
If the null subspace of $K$ contains only the zero element, i.e., $Ker(K)=\{0\}$, then the inequality stated above holds strictly for $h\neq 0$. This implication indicates that $x_{\alpha}^{\delta}$ in \eqref{equ3.4} is the unique global minimum solution of $\Phi_{\alpha}(x)$.
\end{remark}  

\par Although Theorem \ref{the3.3} specifies that the operator $K$ is diagonal under a specific orthogonal bases, it suggests a new SVD method for a more general $K$. As $\alpha\rightarrow 0$, the nonlinear soft thresholding function $\mathcal{S}_{\alpha}$ in \eqref{equ3.4} degenerates into an identity operator and \eqref{equ3.4} degenerates to \eqref{equ2.4}. Therefore, we further prove \eqref{equ3.4} is a regularization strategy, whether the operator $K$ is diagonal or not. To demonstrate the regularization properties of \eqref{equ3.4}, we first recall the concept of a regularization strategy, for details see \cite[Chapter 2]{K1996}.

\begin{definition}\label{def3.5}
    A regularization strategy is defined as a family of linear and bounded operators
    \begin{equation*}
        \mathcal{R}_{\alpha}:Y\rightarrow X\ ,\ 0<\alpha<1
    \end{equation*}
    such that
    \begin{equation*}
        \lim_{\alpha\rightarrow 0}\mathcal{R}_{\alpha}Kx=x\ ,\ for\ all\ x\in X.
    \end{equation*}
\end{definition}

\begin{theorem}\label{the3.6}
    Let $K: X\rightarrow Y$ be a linear compact operator with a singular system denoted by $\left\{\left(\sigma_{n},v_{n},u_{n}\right)\right\}_{n\in\mathbb{N}}$. For every $\alpha>0$ and $0<\sigma_{n}\leq\|K\|$, we suppose there exists $c(\alpha)$ such that
    \begin{equation*}
        c(\alpha)\sigma_{n}\geq 1,
    \end{equation*}
    where $\sigma_{n}$ satisfies $\left|\sigma_{n}\left<y^{\delta}, u_{n}\right>\right|>\frac{\alpha}{2}$. Then the operator $\mathcal{R}_{\alpha}:Y\rightarrow \ell_{2}$, $\alpha>0$, defined by
    \begin{equation*}
        \mathcal{R}_{\alpha}y^{\delta}:=\sum_{n\in\mathbb{N}}\frac{1}{\sigma_{n}^2}\mathcal{S}_{\alpha}(\sigma_{n}\left<y^{\delta},u_{n}\right>)v_{n}\ ,\ y^{\delta}\in Y
    \end{equation*}
    constitutes a regularization strategy with $\|\mathcal{R}_{\alpha}\|\leq c(\alpha)$.
\end{theorem}

\begin{proof}
    We begin by establishing the boundedness of the operator $\mathcal{R}_{\alpha}$. We define the sets $\mathbb{N}_{1}=\{n\in\mathbb{N}\mid\sigma_{n}\left<y^{\delta},u_{n}\right>\leq -\frac{\alpha}{2}\}$, $\mathbb{N}_{2}=\{n\in\mathbb{N}\mid \sigma_{n}\left<y^{\delta},u_{n}\right>\geq \frac{\alpha}{2}\}$ and $\mathbb{N}_{3}=\{n\in\mathbb{N}\mid |\sigma_{n}\left<y^{\delta},u_{n}\right>|<\frac{\alpha}{2}\}$. It is evident that $\mathbb{N}=\mathbb{N}_{1}\cup\mathbb{N}_{2}\cup\mathbb{N}_{3}$. We have  
    \begin{equation*}
        \|\mathcal{R}_{\alpha}y^{\delta}\|^2=\sum_{n\in\mathbb{N}_{1}\cup\mathbb{N}_{2}\cup\mathbb{N}_{3}}\left[\frac{1}{\sigma_{n}^{2}}\mathcal{S}_{\alpha}(\sigma_{n}\left<y^{\delta},u_{n}\right>)\right]^2.
    \end{equation*}
    For $n\in\mathbb{N}_{1}$, we have
    \begin{align}\label{equ3.8}
        \sum_{n\in\mathbb{N}_{1}}\left[\frac{1}{\sigma_{n}^{2}}\mathcal{S}_{\alpha}(\sigma_{n}\left<y^{\delta},u_{n}\right>)\right]^2
        & =\sum_{n\in\mathbb{N}_{1}}\frac{1}{\sigma_{n}^{4}}\left[(\sigma_{n}\left<y^{\delta},u_{n}\right>)+\frac{\alpha}{2}\right]^2 \nonumber\\
        & =\sum_{n\in\mathbb{N}_{1}}\left[\frac{|\left<y^{\delta},u_{n}\right>|^2}{\sigma_{n}^{2}}+\frac{\alpha\left<y^{\delta},u_{n}\right>}{\sigma_{n}^{3}}+\frac{\alpha^{2}}{4\sigma_{n}^4}\right].
    \end{align}
    Since $n\in\mathbb{N}_{1}$, we know that
    \begin{equation*}
        \left<y^{\delta},u_{n}\right>\leq-\frac{\alpha}{2\sigma_{n}}.
    \end{equation*}
    Consequently, we get
    \begin{equation*}
       \frac{\alpha\left<y^{\delta},u_{n}\right>}{\sigma_{n}^{3}}+\frac{\alpha^{2}}{4\sigma_{n}^4}\leq 0.
    \end{equation*}
    Thus, equation \eqref{equ3.8} leads to the conclusion that
    \begin{equation}\label{equ3.9}
        \sum_{n\in\mathbb{N}_{1}}\left[\frac{1}{\sigma_{n}^{2}}\mathcal{S}_{\alpha}(\sigma_{n}\left<y^{\delta},u_{n}\right>)\right]^2 \leq\sum_{n\in\mathbb{N}_{1}}\frac{1}{\sigma_{n}^2}|\left<y^{\delta},u_{n}\right>|^2.
    \end{equation}
    For $n\in\mathbb{N}_{2}$, we find
    \begin{align}\label{equ3.10}
         \sum_{n\in\mathbb{N}_{2}}\left[\frac{1}{\sigma_{n}^{2}}\mathcal{S}_{\alpha}(\sigma_{n}\left<y^{\delta},u_{n}\right>)\right]^2
         & =\sum_{n\in\mathbb{N}_{2}}\frac{1}{\sigma_{n}^{4}}\left[(\sigma_{n}\left<y^{\delta},u_{n}\right>)-\frac{\alpha}{2}\right]^2 \nonumber\\
         & =\sum_{n\in\mathbb{N}_{2}}\frac{1}{\sigma_{n}^{4}}\left[(\sigma_{n}^{2}|\left<y^{\delta},u_{n}\right>|^{2})-\alpha\sigma_{n}\left<y,u_{n}\right>+\frac{\alpha^{2}}{4}\right]. 
    \end{align}
    For $n\in\mathbb{N}_{2}$, we can infer that
    \begin{equation*}
        \left<y^{\delta},u_{n}\right>\geq\frac{\alpha}{2\sigma_{n}}.
    \end{equation*}
    Consequently,
    \begin{equation*}
        -\alpha\sigma_{n}\left<y^{\delta},u_{n}\right>+\frac{\alpha^{2}}{4}\leq 0.
    \end{equation*}
    Therefore, equation \eqref{equ3.10} yields that
    \begin{equation}\label{equ3.11}
        \sum_{n\in\mathbb{N}_{2}}\left[\frac{1}{\sigma_{n}^{2}}\mathcal{S}_{\alpha}(\sigma_{n}\left<y^{\delta},u_{n}\right>)\right]^2 \leq\sum_{n\in\mathbb{N}_{2}}\frac{1}{\sigma_{n}^2}|\left<y^{\delta},u_{n}\right>|^2.
    \end{equation}
    For $n\in\mathbb{N}_{3}$, we obtain
    \begin{equation}\label{equ3.12}
       \sum_{n\in\mathbb{N}_{3}}\left[\frac{1}{\sigma_{n}^{2}}\mathcal{S}_{\alpha} (\sigma_{n}\left<y^{\delta},u_{n}\right>)\right]^2=0.
    \end{equation}
    By combining \eqref{equ3.9}, \eqref{equ3.11} and \eqref{equ3.12}, we derive
    \begin{align*}
        \|\mathcal{R}_{\alpha}y^{\delta}\|^{2}
        & = \sum_{n\in\mathbb{N}_{1}\cup\mathbb{N}_{2}\cup\mathbb{N}_{3}}\left[\frac{1}{\sigma_{n}^{2}}\mathcal{S}_{\alpha}(\sigma_{n}\left<y^{\delta},u_{n}\right>)\right]^2 \\ 
        & = \sum_{n\in\mathbb{N}_{1}}\left[\frac{1}{\sigma_{n}^{2}}\mathcal{S}_{\alpha}(\sigma_{n}\left<y^{\delta},u_{n}\right>)\right]^2+\sum_{n\in\mathbb{N}_{2}}\left[\frac{1}{\sigma_{n}^{2}}\mathcal{S}_{\alpha}(\sigma_{n}\left<y^{\delta},u_{n}\right>)\right]^2+\sum_{n\in\mathbb{N}_{3}}\left[\frac{1}{\sigma_{n}^{2}}\mathcal{S}_{\alpha}(\sigma_{n}\left<y^{\delta},u_{n}\right>)\right]^2 \\
        & \leq \sum_{n\in\mathbb{N}_{1}}\frac{1}{\sigma_{n}^2}|\left<y^{\delta},u_{n}\right>|^2+ \sum_{n\in\mathbb{N}_{2}}\frac{1}{\sigma_{n}^2}|\left<y^{\delta},u_{n}\right>|^2+0 \\
        & = \sum_{n\in\mathbb{N}_{1}\cup\mathbb{N}_{2}}\frac{1}{\sigma_{n}^2}|\left<y^{\delta},u_{n}\right>|^2 \leq c(\alpha)^{2}\|y^{\delta}\|^{2}.
    \end{align*}
    Thus, we conclude that $\|\mathcal{R}_{\alpha}\|\leq c(\alpha)$.
    
    \par Next, we aim to demonstrate the convergence of $\mathcal{R}_{\alpha}Kx$ to $x$ as $\alpha\rightarrow 0$. By invoking Lemma \ref{lem2.6}, we have $\sigma_{n}\left<y,u_{n}\right>=\sigma_{n}^{2}\left<x,v_{n}\right>$. For every $\epsilon>0$, there exist an $N\in\mathbb{N}$ such that for any fixed $x\in X$,
    \begin{equation*}
        \sum_{n=N+1}^{\infty}|\left<x,v_{n}\right>|^{2}\leq\frac{\epsilon^{2}}{2}.
    \end{equation*}
    As $\alpha\rightarrow 0$, for all $i=1,\dots,N$, we derive the following inequality
    \begin{equation*}
        |\alpha-0|^{2}\leq \frac{2\sigma_{i}^{4}}{N}\epsilon^{2} \leq \frac{2\|K\|^{4}}{N}\epsilon^{2}.
    \end{equation*}
     For $n\in\mathbb{N}_{1}$, where $\sigma_{n}^{2}\left<x,v_{n}\right> >\frac{\alpha}{2}$, we have
    \begin{align}\label{equ3.13}
        \sum_{n\in\mathbb{N}_{1}}\left\|\frac{1}{\sigma_{n}^2}\mathcal{S}_{\alpha}(\sigma_{n}\left<Kx,u_{n}\right>)v_{n}-\left<x,v_{n}\right>v_{n}\right\|^2
        & = \sum_{n\in\mathbb{N}_{1}}\left\|\frac{1}{\sigma_{n}^2}(\sigma_{n}^{2}\left<x,v_{n}\right>v_{n}+\frac{\alpha}{2}v_{n})-\left<x,v_{n}\right>v_{n}\right\|^{2} \nonumber\\
        & =\sum_{n\in\mathbb{N}_{1}}\frac{\alpha^{2}}{4\sigma_{n}^{4}} \nonumber\\
        & \leq \sum_{\{n\in\mathbb{N}_{1}\}\cap\{n\leq N\}}\frac{\alpha^{2}}{4\sigma_{n}^{4}}+\sum_{\{n\in\mathbb{N}_{1}\}\cap \{n> N\}}|\left<x,v_{n}\right>|^{2}.
    \end{align}
    For $n\in\mathbb{N}_{2}$, where $\sigma_{n}^{2}\left<x,v_{n}\right> <-\frac{\alpha}{2}$, we find
    \begin{align}\label{equ3.14}
        \sum_{n\in\mathbb{N}_{2}}\left\|\frac{1}{\sigma_{n}^2}\mathcal{S}_{\alpha}(\sigma_{n}\left<Kx,u_{n}\right>)v_{n}-\left<x,v_{n}\right>v_{n}\right\|^2
        & = \sum_{n\in\mathbb{N}_{2}}\left\|\frac{1}{\sigma_{n}^2}(\sigma_{n}^{2}\left<x,v_{n}\right>v_{n}-\frac{\alpha}{2}v_{n})-\left<x,v_{n}\right>v_{n}\right\|^2 \nonumber\\
        & =\sum_{n\in\mathbb{N}_{2}}\frac{\alpha^{2}}{4\sigma_{n}^{4}} \nonumber\\
        & \leq \sum_{\{n\in\mathbb{N}_{2}\}\cap\{n\leq N\}}\frac{\alpha^{2}}{4\sigma_{n}^{4}}+\sum_{\{n\in\mathbb{N}_{2}\}\cap \{n> N\}}|\left<x,v_{n}\right>|^{2}.
    \end{align}
    For $n\in\mathbb{N}_{3}$, where $|\sigma_{n}^{2}\left<x,v_{n}\right>| \leq\frac{\alpha}{2}$, we derive
    \begin{align}\label{equ3.15}
        \sum_{n\in\mathbb{N}_{3}}\left\|\frac{1}{\sigma_{n}^2}\mathcal{S}_{\alpha}(\sigma_{n}\left<Kx,u_{n}\right>)v_{n}-\left<x,v_{n}\right>v_{n}\right\|^2
        & =\sum_{n\in\mathbb{N}_{3}}\left\|\left<x,v_{n}\right>v_{n}\right\|^{2}=\sum_{n\in\mathbb{N}_{3}}|\left<x,v_{n}\right>|^2 \nonumber\\
        & \leq \sum_{\{n\in\mathbb{N}_{3}\}\cap\{n\leq N\}}\frac{\alpha^{2}}{4\sigma_{n}^{4}}+\sum_{\{n\in\mathbb{N}_{3}\}\cap \{n> N\}}|\left<x,v_{n}\right>|^{2}.
    \end{align} 
    By combining equations \eqref{equ3.13}, \eqref{equ3.14}and \eqref{equ3.15}, we obtain the following result
    \begin{align*}
        \|\mathcal{R}_{\alpha}Kx-x\|^2
        & = \sum_{n\in\mathbb{N}_{1}\cup\mathbb{N}_{2}\cup\mathbb{N}_{3}} \left\|\frac{1}{\sigma_{n}^2}\mathcal{S}_{\alpha}(\sigma_{n}\left<Kx,u_{n}\right>)v_{n}-\left<x,v_{n}\right>v_{n}\right\|^2 \nonumber\\
        & \leq \left[\sum_{\{n\in\mathbb{N}_{1}\}\cap\{n\leq N\}}\frac{\alpha^{2}}{4\sigma_{n}^{4}}+\sum_{\{n\in\mathbb{N}_{2}\}\cap \{n\leq N\}}\frac{\alpha^{2}}{4\sigma_{n}^{4}}+\sum_{\{n\in\mathbb{N}_{3}\}\cap \{n\leq N\}}\frac{\alpha^{2}}{4\sigma_{n}^{4}}\right] \nonumber\\
        & \quad +\left[\sum_{\{n\in\mathbb{N}_{1}\}\cap\{n> N\}}|\left<x,v_{n}\right>|^{2}+\sum_{\{n\in\mathbb{N}_{2}\}\cap\{n> N\}}|\left<x,v_{n}\right>|^{2}+\sum_{\{n\in\mathbb{N}_{3}\}\cap\{n> N\}}|\left<x,v_{n}\right>|^{2}\right] \nonumber\\
        & = \sum_{n=1}^{N}\frac{\alpha^{2}}{4\sigma_{n}^{4}}+\sum_{n=N+1}^{\infty}|\left<x,v_{n}\right>|^{2} \nonumber\\
        & \leq \sum_{n=1}^{N}\frac{2\sigma_{n}^{4}}{N}\epsilon^{2} \frac{1}{4\sigma_{n}^{4}}+\frac{\epsilon^{2}}{2}= \epsilon^{2}.
    \end{align*}
    Thus, since $\{\sigma_{n}\}_{n\in\mathbb{N}}$ represents the ordered sequence of the positive singular values of $K$, we have demonstrated that for every $x\in X$,
    \begin{equation*}
        \mathcal{R}_{\alpha}Kx\rightarrow x\ (\alpha\rightarrow 0).
    \end{equation*}
    This confirms the proof of the theorem.
\end{proof}

\par The algorithm for SVD that employs a soft thresholding function is outlined in Algorithm 1. Subsequently, we delineate a property of the soft threshold function \cite[Lemma 2.2]{DDD04} and present the error estimates associated with this regularization strategy.

\begin{lemma}\label{lem3.7}
    The soft thresholding function $\mathcal{S}_{\alpha}$ is nonexpansive, i.e., for $\forall\ t,t'\in\mathbb{R}$
    \begin{equation*}
        \left|\mathcal{S}_{\alpha}\left(t\right)-\mathcal{S}_{\alpha}\left(t'\right)\right| \leq \left|t-t'\right|.
    \end{equation*} 
\end{lemma}

\begin{algorithm}
\caption{SVD algorithm for the $\ell^{1}$-SVD regularization strategy}
\begin{algorithmic}\label{alg1}
\STATE{\textbf{Input:} Given regularization parameter $\alpha>0$, 
observation data $y^{\delta}$ and linear operator $K$.}
\STATE{\qquad $[U^{T},\Sigma,V]$ = SVD($K$)}
\STATE{\qquad Let $U^{T}=[u_{1},u_{2},\cdots,u_{m}]$, $V^{T}=[v_{1},v_{2},\cdots,v_{n}]$, and $\sigma_{i}=\Sigma(i,i)$.}
\STATE{\qquad For $i=1\cdots \min\{m,n\}$ :}
\STATE{\qquad\qquad $x_{i}^{\delta}=(\mathcal{S}_{\alpha}(\sigma_{i}\left<y^{\delta}, u_{i}\right>)/ \sigma_{i}^{2})v_{i}$}
\STATE{\qquad end}
\STATE{\qquad $x_{\alpha}^{\delta}=\sum_{i=1}^{\min\{m,n\}}(x_{i}^{\delta})$}

\STATE{\textbf{Output:} The regularized solution $x_{\alpha}^{\delta}$.}
\end{algorithmic}
\end{algorithm}

\begin{theorem}\label{the3.8}
   Let assumption of the Theorem \ref{the3.6} hold. \\
   (1)There exists $c_{1}\geq\sigma_{n}/\sqrt{\alpha}>0$ for all $\alpha>0$ and $0<\sigma_{n}\leq\|K\|$. If $x\in\mathcal{R}(K^{*})$, then
   \begin{equation}\label{3.16}
       \|\mathcal{R}_{\alpha}Kx-x\|\leq c_{1}\sqrt{\alpha}\|z\|,
   \end{equation}
   where $x=K^{*}z$. \\
   (2)There exists $c_{2}\geq\sigma_{n}^{2}/\alpha>0$ with for all $\alpha>0$ and $0<\sigma_{n}\leq\|K\|$. If $x\in\mathcal{R}(K^{*}K)$, then
   \begin{equation}\label{equ3.17}
       \|\mathcal{R}_{\alpha}Kx-x\|\leq c_{2}\alpha\|z\|,
   \end{equation}
   where $x=K^{*}Kz$.
\end{theorem}

\begin{proof}
    With $x=K^{*}z$ and $\left<x,v_{n}\right>=\sigma_{n}\left<z,u_{n}\right>$, When $\sigma_{n}^3\left<z,u_{n}\right> < -\frac{\alpha}{2}$ and $\sigma_{n}^3\left<z,u_{n}\right> > \frac{\alpha}{2}$,
    formula \eqref{equ3.13} and \eqref{equ3.14} takes the form
    \begin{equation*}
        \sum_{n\in\mathbb{N}_{1}}\frac{\alpha^{2}}{4\sigma_{n}^{4}}\leq \sum_{n\in\mathbb{N}_{1}}\sigma_{n}^{2}\left| \left<z,u_{n}\right>\right|\quad {\rm and} \quad \sum_{n\in\mathbb{N}_{2}}\frac{\alpha^{2}}{4\sigma_{n}^{4}}\leq \sum_{n\in\mathbb{N}_{2}}\sigma_{n}^{2}\left| \left<z,u_{n}\right>\right|.
    \end{equation*}
    When $|\sigma_{n}^3\left<z,u_{n}\right>|\leq \frac{\alpha}{2}$, formula \eqref{equ3.15} takes the form
    \begin{equation*}
        \sum_{n\in\mathbb{N}_{3}}|\left<x,v_{n}\right>|^2= \sum_{n\in\mathbb{N}_{3}}\sigma_{n}^{2}|\left<z,u_{n}\right>|^2.
    \end{equation*}
    \par In summary, we obtain
    \begin{align*}
        \|\mathcal{R}_{\alpha}Kx-x\|^2
        & =\sum_{n\in\mathbb{N}_{1}}\frac{\alpha^{2}}{4\sigma_{n}^{4}}+ \sum_{n\in\mathbb{N}_{2}}\frac{\alpha^{2}}{4\sigma_{n}^{4}}+ \sum_{n\in\mathbb{N}_{3}}|\left<x,v_{n}\right>|^2  \\
        & \leq \sum_{n\in\mathbb{N}_{1}\cup\mathbb{N}_{2}\cup\mathbb{N}_{3}} \sigma_{n}^{2}|\left<z,u_{n}\right>|^2 \\
        & = \sum_{n\in\mathbb{N}} \sigma_{n}^{2}|\left<z,u_{n}\right>|^2 \leq (c_{1}\sqrt{\alpha}\|z\|)^{2}.
    \end{align*}
    The first case is proved and the second case can be proved by similar methods.
\end{proof}

\begin{theorem}\label{the3.9}
    Let $y^{\delta}\in Y$ satisfy $\|y^{\delta}-y\|\leq\delta$, where $y=Kx$ denotes the exact data and $0<\alpha < 1$. Let $K:X\rightarrow Y$ be a compact operator with a singular system $\left\{\sigma_{n},v_{n},u_{n}\right\}_{n\in\mathbb{N}}$. The operator
    \begin{equation*}
        \mathcal{R}_{\alpha}y^{\delta}:=\sum_{n\in\mathbb{N}}\frac{1}{\sigma_{n}^2}\mathcal{S}_{\alpha}(\sigma_{n}\left<y^{\delta}, u_{n}\right>)v_{n}\ ,\ y^{\delta}\in Y.
    \end{equation*}
    is a regularization operator, and $\left\|\mathcal{R}_{\alpha}\right\|\leq c\left(\alpha\right)$. Define $x^{\alpha,\delta}:=\mathcal{R}_{\alpha}y^{\delta}$ as an approximation of the solution $x$ of $Kx=y$. \\
    (1) Let $x=K^{*}z\in\mathcal{R}(K^{*})$ with $\|z\|\leq E$ and $c>0$. For the choice $\alpha(\delta)=c(\delta/E)^{4/3}$, we have the estimate
    \begin{equation}\label{equ3.18}
        \|x^{\alpha(\delta),\delta}-x\|\leq(\frac{d_{1}}{c}+\sqrt{d_{2}c})\delta^{1/3} E^{2/3}.
    \end{equation} \\
    (2) Let $x=K^{*}Kz\in\mathcal{R}(K^{*}K)$ with $\|z\|\leq E$ and $c>0$. The choice $\alpha(\delta)=c\delta/E$  the estimate
    \begin{equation}\label{equ3.19}
        \|x^{\alpha(\delta),\delta}-x\|\leq(\frac{d_{1}}{c}+d_{2}c)\delta^{1/2} E^{1/2}.
    \end{equation}
    Therefore, the approximate solution is optimal for the information $\|(K^{*})^{-1}x\|\leq E$ or $\|(K^{*}K))^{-1}\|\leq E$, respectively(if $K^{*}$ is one-to-one).
\end{theorem}

\begin{proof}
    Let $c\left(\alpha\right)\sigma_{n}\geq 1$, where $\sigma_{n}$ satisfies $\left|\sigma_{n}\left<y^{\delta}, u_{n}\right>\right|>\frac{\alpha}{2}$ or $\left|\sigma_{n}\left<y, u_{n}\right>\right|>\frac{\alpha}{2}$. The error splits into two parts by the following obvious application of the triangle inequality:
    \begin{equation*}
        \|x^{\alpha,\delta}-x\|\leq\|\mathcal{R}_{\alpha}y^{\delta}-\mathcal{R}_{\alpha}y\| +\|\mathcal{R}_{\alpha}y-x\|.
    \end{equation*}
    We have
    \begin{equation*}
        \left\|\mathcal{R}_{\alpha}y^{\delta}-\mathcal{R}_{\alpha}y\right\|^{2}= \sum_{n\in\mathbb{N}}\frac{1}{\sigma_{n}^4}\left\|\mathcal{S}_{\alpha}(\sigma_{n}\left<y^{\delta},u_{n}\right>)v_{n} -\mathcal{S}_{\alpha}(\sigma_{n}\left<y,u_{n}\right>)v_{n}\right\|^{2}.
    \end{equation*}
    According to Lemma \ref{lem3.7}, we will examine the distinct scenarios concerning $y^{\delta}$ and $y$ in a systematic manner, addressing each case individually. \\
    Case 1, if $\left|\sigma_{n}\left<y^{\delta}, u_{n}\right>\right|>\frac{\alpha}{2}$ and $\left|\sigma_{n}\left<y, u_{n}\right>\right|>\frac{\alpha}{2}$, then
    \begin{align*}
        \sum\frac{1}{\sigma_{n}^4}\left\|\mathcal{S}_{\alpha}(\sigma_{n}\left<y^{\delta},u_{n}\right>)v_{n} -\mathcal{S}_{\alpha}(\sigma_{n}\left<y,u_{n}\right>)v_{n}\right\|^{2}
        & \leq \sum\frac{1}{\sigma_{n}^{2}}\left\|\left<y^{\delta}-y,u_{n}\right>v_{n} \right\|^{2} \nonumber\\
        & \leq c\left(\alpha\right)^{2}\sum\left|\left<y^{\delta}-y,u_{n}\right>\right|^{2} \\
        & \leq c\left(\alpha\right)^{2}\delta^{2}.
    \end{align*}
    Case 2, if $\left|\sigma_{n}\left<y^{\delta}, u_{n}\right>\right|\leq\frac{\alpha}{2}$ and $\left|\sigma_{n}\left<y, u_{n}\right>\right|\leq\frac{\alpha}{2}$, then
    \begin{equation*}
        \sum\frac{1}{\sigma_{n}^4}\left\|\mathcal{S}_{\alpha}(\sigma_{n}\left<y^{\delta},u_{n}\right>)v_{n} -\mathcal{S}_{\alpha}(\sigma_{n}\left<y,u_{n}\right>)v_{n}\right\|^{2}=0.
    \end{equation*}
    Case 3, if $\left|\sigma_{n}\left<y^{\delta}, u_{n}\right>\right|>\frac{\alpha}{2}$ and $\left|\sigma_{n}\left<y, u_{n}\right>\right|\leq\frac{\alpha}{2}$, then
    \begin{align*}
        \sum\frac{1}{\sigma_{n}^4}\left\|\mathcal{S}_{\alpha}(\sigma_{n}\left<y^{\delta},u_{n}\right>)v_{n} -\mathcal{S}_{\alpha}(\sigma_{n}\left<y,u_{n}\right>)v_{n}\right\|^{2}
        &  = \sum \frac{1}{\sigma_{n}^4}\left\|\left(\sigma_{n}\left|\left<y^{\delta},u_{n}\right>\right|- \frac{\alpha}{2}\right)v_{n}\right\|^{2} \\
        & \leq \sum \frac{1}{\sigma_{n}^2}\left\|\left(\left|\left<y^{\delta},u_{n}\right>\right|- \left|\left<y, u_{n}\right>\right|\right)v_{n}\right\|^{2} \\
        & \leq c\left(\alpha\right)^{2}\sum \left|\left<y^{\delta}-y,u_{n}\right>\right|^{2} \\
        & \leq c\left(\alpha\right)^{2}\delta^{2}
    \end{align*}
    Case 4, if $\left|\sigma_{n}\left<y^{\delta}, u_{n}\right>\right|\leq\frac{\alpha}{2}$ and $\left|\sigma_{n}\left<y, u_{n}\right>\right|>\frac{\alpha}{2}$, then
    \begin{align*}
        \sum\frac{1}{\sigma_{n}^4}\left\|\mathcal{S}_{\alpha}(\sigma_{n}\left<y^{\delta},u_{n}\right>)v_{n} -\mathcal{S}_{\alpha}(\sigma_{n}\left<y,u_{n}\right>)v_{n}\right\|^{2}
        &  = \sum \frac{1}{\sigma_{n}^4}\left\|\left(\sigma_{n}\left|\left<y,u_{n}\right>\right|- \frac{\alpha}{2}\right)v_{n}\right\|^{2} \\
        & \leq \sum \frac{1}{\sigma_{n}^2}\left\|\left(\left|\left<y,u_{n}\right>\right|- \left|\left<y^{\delta}, u_{n}\right>\right|\right)v_{n}\right\|^{2} \\
        & \leq c\left(\alpha\right)^{2} \sum \left|\left<y-y^{\delta},u_{n}\right>\right|^{2} \\
        & \leq c\left(\alpha\right)^{2}\delta^{2}
    \end{align*}
    Through the discussion above, we obtain
    \begin{equation}\label{equ3.20}
        \left\|\mathcal{R}_{\alpha}y^{\delta}-\mathcal{R}_{\alpha}y\right\|^{2}\leq c\left(\alpha\right)^{2}\delta^{2}
    \end{equation}
    From the condition 
    \begin{equation}\label{con1.1}
        \left|\sigma_{n}\left<y^{\delta}, u_{n}\right>\right|>\frac{\alpha}{2}\quad {\rm or}\quad \left|\sigma_{n}\left<y, u_{n}\right>\right|>\frac{\alpha}{2}
    \end{equation}
    with $0<\alpha<1$, we obtain
    \begin{equation*}
        \sigma_{n}\frac{2\left|\left<y^{\delta},u_{n}\right>\right|}{\alpha}> 1\quad {\rm or} \quad \sigma_{n}\frac{2\left|\left<y,u_{n}\right>\right|}{\alpha}> 1.
     \end{equation*}
    So we define $d_{1}:=\max\left\{2\left|\left<y^{\delta},u_{n}\right>\right|, 2\left|\left<y,u_{n}\right>\right|\right\}$ for $n$ satisfying the condition \ref{con1.1}. By formula \eqref{equ3.14}, \eqref{equ3.15} and Theorem \ref{the3.8}, when
    \begin{equation}\label{con1.2}
        \left|\sigma_{n}^{3}\left<z, u_{n}\right>\right|>\frac{\alpha}{2},
    \end{equation}
    we also obtain
    \begin{align*}
        \sum_{n\in\mathbb{N}_{1}\cup\mathbb{N}_{2}}\frac{\alpha^{2}}{4\sigma_{n}^{4}}
        \leq \left(c_{1}\sqrt{\alpha}\right)^{2}\left\|z\right\|^{2}.
    \end{align*} 
    So we define $d_{2}:=\max\left\{n/\left(2\sigma_{n}\sum_{i=1}^{n}\left|\left<z,u_{i}\right>\right|\right) \right\}$ for $n$ satisfying the condition \ref{con1.2}. 
    \par Furthermore, we choose $c(\alpha)=d_{1}/\alpha$, $c_{1}=\sqrt{d_{2}}$ and $c_{2}=d_{2}$.
    Combining \eqref{equ3.20} with Theorem \ref{the3.8} yields the error estimate
    \begin{equation*}
        \|x^{\alpha,\delta}-x\|\leq\frac{d_{1}\delta}{\alpha}+\sqrt{d_{2}\alpha}\|z\|
    \end{equation*}
    and
    \begin{equation*}
        \|x^{\alpha,\delta}-x\|\leq\frac{d_{1}\delta}{\alpha}+d_{2}\alpha\|z\|.
    \end{equation*}
    The choices $\alpha(\delta)=c\left(\delta/E\right)^{2/3}$ and $\alpha(\delta)=c\left(\delta/E\right)^{1/2}$ lead to the estimates \eqref{equ3.18} and \eqref{equ3.19}, respectively. This confirms the proof of the theorem.
\end{proof}

\begin{remark}
    Unlike the linear threshold operator, the nonlinear regularization operator filters out smaller singular values that do not meet the condition $\left|\sigma_{n}\left<y^{\delta}, u_{n}\right>\right|>\frac{\alpha}{2}$ through a thresholding process to achieve regularization. In this case, the corresponding regularization parameter $\alpha$ is implicit within the threshold condition. Therefore, we select $c\left(\alpha\right)$ based on this threshold criterion.
\end{remark}

\subsection{SVD method for $\ell^{1/2}$-SVD regularization operator}\label{sec3.2}

\par In this section, we explore the interplay between the SVD method and sparsity regularization with $\ell^{1/2}$ penalties, formally defined by the optimization problem
\begin{equation}\label{equ3.21}
    \min_{x\in\ell_{2}}\left\{\Phi_{\alpha}(x)=\left\|Kx-y^{\delta}\right\|_{Y}^{2}+ \alpha\sum_{n\in\mathbb{N}}\left|\left<x,v_{n}\right>\right|^{1/2}\right\},
\end{equation}
where $\alpha>0$ is the regularization parameter. In \cite{XCXZ12}, an iterative half-thresholding algorithm is proposed to rapidly solve the sparsity regularization with $\ell^{1/2}$ penalties. In scenarios where the operator $K$ is diagonal within a specific orthogonal basis, we present a direct formulation of the SVD calculation pertinent to the corresponding regularization solution, accompanied by a detailed proof. Under this condition, we adapt the SVD methodology to effectively demonstrate its capacity to minimize the functional incorporation of the non-smooth regularization term.

\par We begin by defining the half-thresholding function and identifying the stationary points of the functional $\Phi_{\alpha}\left(x\right)$ as described in \eqref{equ3.21}.

\begin{definition}\label{def3.10}
    For $t\in\mathbb{R}$, $\alpha>0$, and a singular system $\{(\sigma_{n},u_{n},v_{n})\}_{n\in\mathbb{N}}$, we define the function $\mathcal{H}_{\alpha,n}$ as follows:
    \begin{align}\label{equ3.22}
        \mathcal{H}_{\alpha,n}(t)=
        \left\{             
        \begin{array}{ll}
            \frac{2}{3\sigma_{n}^{4/3}}t\left(1+cos\left(\frac{2}{3}\pi-\frac{2}{3}\phi_{\alpha}(t)\right)\right), & |t| > \frac{3}{4}\alpha^{2/3}, \\
            0, & |t| \leq \frac{3}{4}\alpha^{2/3},
        \end{array}
        \right.
    \end{align}
    where
    \begin{equation*}
        \phi_{\alpha}(t)=arccos\left(\frac{\alpha}{8}\left(\frac{|t|}{3}\right)^{-3/2}\right).
    \end{equation*}
    Thus, $\mathcal{H}_{\alpha,n}$ is referred to as a half thresholding function from $\mathbb{R}$ to $\mathbb{R}$.
\end{definition}

\begin{definition}\label{def3.11}
    \textbf{(Stationary point)} A point $x$ is considered a stationary point of $\Phi_{\alpha}\left(x\right)$ as defined in \eqref{equ3.21}, if it satisfies the following first-order optimization condition
    \begin{equation}\label{equ3.23}
        2K^{*}\left(Kx-y^{\delta}\right)+\alpha\nabla\left(\sum_{n\in\mathbb{N}} \left|\left<x,v_{n}\right>\right|^{1/2}\right) =0.
    \end{equation}
\end{definition}

\begin{theorem}\label{the3.12}
    Let $K: X\rightarrow Y$ be a linear compact operator and $\{(\sigma_{n},u_{n},v_{n})\}_{n\in\mathbb{N}}$ be a singular system of $K$. If $K$ happens to be diagonal in the $v_{n}$-basis and a minimum solution $x_{\alpha}^{\delta}$ exists for $\Phi_{\alpha}\left(x\right)$ as defined in \eqref{equ3.10}, it can be expressed as
    \begin{equation}\label{equ3.24}
        x_{\alpha}^{\delta}=\sum_{n\in\mathbb{N}}\mathcal{H}_{\alpha,n}(\sigma_{n}^{1/3}\left<y^{\delta},u_{n}\right>)v_{n},
    \end{equation}
    where $\mathcal{H}_{\alpha,n}$ is the half thresholding function described in  Definition \ref{def3.10} and $y^{\delta}_{n}$ represents $\left<y^{\delta},u_{n}\right>$. Furthermore, if $x_{\alpha}^{\delta}$ can be represented in the form of \eqref{equ3.24}, then it qualifies as a stationary point of the functional $\Phi_{\alpha}\left(x\right)$ in \eqref{equ3.21}. Additionally, it is a locally optimal solution in $U\left(x_{\alpha}^{\delta},\rho\right)$ with $0<\rho<\sum_{n\notin\Lambda}\alpha^{2} /(4\sigma_{n}^{2} {y^{\delta}_{n}}^{2})$, provided that the following conditions are met
    \begin{equation*}
       \left\|x_{\alpha}^{\delta}\right\|_{\ell_{2}}>C\quad and\quad c(\alpha)\sigma_{n}\geq 1,
    \end{equation*}
    for $n\in\mathbb{N}$, where $\sigma_{n}$ satisfies $\left|\sigma_{n}^{1/3}\left<y^{\delta},u_{n}\right>\right|>\frac{3}{4}\alpha^{2/3}$, and $c(\alpha)=2\sqrt{2}C^{3/4}$.
\end{theorem}

\begin{proof}
    We note that $\sum_{n\in\mathbb{N}}\left|\left<x,v_{n}\right>\right|^{1/2}$ is continuously differentiable for all $x\in X$ with $x\neq 0$. Its gradient is given by
    \begin{equation*}
        \nabla\left(\sum_{n\in\mathbb{N}}\left|\left<x,v_{n}\right>\right|^{1/2}\right)= \left(\frac{{\rm sign} (x_{1})}{2\sqrt{|x_{1}|}}, \frac{{\rm sign} (x_{2})}{2\sqrt{|x_{2}|}}, \cdots, \frac{{\rm sign} (x_{n})}{2\sqrt{|x_{n}|}}\right)^{T}.
    \end{equation*}
    It is evident that \eqref{equ3.21} is differentiable for $x_{n}\neq 0$. By applying Lemma \ref{lem2.5} and the operator $K$ is diagonal in the $v_{n}$-basis, $\Phi_{\alpha}\left(x\right)$ in \eqref{equ3.21} can be reformulated as
    \begin{align*}
        \Phi_{\alpha}(x)
        & =\|Kx-y^{\delta}\|_{Y}^{2}+\alpha\sum_{n\in\mathbb{N}} \sqrt{|\left<x,v_{n}\right>|} \nonumber\\
        & =\sum_{n\in\mathbb{N}}\left[|\sigma_{n}x_{n}-y^{\delta}_{n}|^{2} +\alpha\sqrt{|x_{n}|}\right],
    \end{align*}
     where $x_{n}$ denotes $\left<x,v_{n}\right>$ and $y^{\delta}_{n}$ signifies $\left<y^{\delta},u_{n}\right>$ for a singular system $\{(\sigma_{n},v_{n},u_{n})\}_{n\in\mathbb{N}}$ of $K$. If $x_{\alpha}^{\delta}$ is a minimum solution of $\Phi_{\alpha}\left(x\right)$ in \eqref{equ3.21}, it must satisfy the first order optimal condition, leading to the variational equation
    \begin{equation}\label{equ3.25}
        2\sigma_{n}(\sigma_{n}\left(x_{\alpha}^{\delta}\right)_{n}-y^{\delta}_{n})+\frac{\alpha {\rm sign} \left(\left(x_{\alpha}^{\delta}\right)_{n}\right)} {2\sqrt{|\left(x_{\alpha}^{\delta}\right)_{n}|}}=0.
    \end{equation}
    
    \par For $\left(x_{\alpha}^{\delta}\right)_{n}\neq 0$ and $\sigma_{n}>0$, we obtain 
    \begin{equation*}
        y^{\delta}_{n}=\sigma_{n} \left(x_{\alpha}^{\delta}\right)_{n}+ \frac{\alpha {\rm sign}\left(\left(x_{\alpha}^{\delta}\right)_{n}\right)} {4\sigma_{n}\sqrt{\left|\left(x_{\alpha}^{\delta}\right)_{n}\right|}}.
    \end{equation*} 
     It is clear that $\left(x_{\alpha}^{\delta}\right)_{n}y^{\delta}_{n}\geq 0$, thus we only need to consider the case where $\left(x_{\alpha}^{\delta}\right)_{n}y^{\delta}_{n}\geq 0$ for $n\in\mathbb{N}$.  
    \par For the case where $\sigma_{n}^{1/3}y^{\delta}_{n}> \frac{3}{4}\alpha^{2/3}$, we define $\eta_{n}:=\sigma_{n}^{2/3}\sqrt{|\left(x_{\alpha}^{\delta}\right)_{n}|}$, and we find that $\sigma_{n}^{4/3}\left(x_{\alpha}^{\delta}\right)_{n}=\eta_{n}^{2}$. Consequently, equation \eqref{equ3.25} can be rewritten as the following cubic equation
    \begin{equation}\label{equ3.26}
        \eta_{n}^{3}-\sigma_{n}^{1/3}y^{\delta}_{n}\eta_{n}+\frac{\alpha}{4}=0.
    \end{equation}
    Using Cartan's root-finding formula, which is expressed in terms of hyperbolic functions (as detailed in \cite{XCXZ12}), we can express the solutions to \eqref{equ3.26}. We denote  
    \begin{equation*}
        r:=\sqrt{\left(\left|\sigma_{n}^{1/3}y^{\delta}_{n}\right|/3\right)},\quad  p:=-\frac{\sigma_{n}^{1/3}y^{\delta}_{n}}{3}\quad {\rm and}\quad q:=\frac{\alpha}{8}.
    \end{equation*}
    Given that $\sigma_{n}^{1/3}y^{\delta}_{n}>(3/4)\alpha^{2/3}$ and setting $\phi=arccos(q/r^{3})$, the three roots of \eqref{equ3.26} can be expressed as follows
    \begin{align*}
        \left\{             
             \begin{array}{ll}
             \eta_{n,1}=-2rcos(\frac{\phi}{3}), \\
             \eta_{n,2}=2rcos(\frac{\pi}{3}+\frac{\phi}{3}), \\
             \eta_{n,3}=2rcos(\frac{\pi}{3}-\frac{\phi}{3}).
             \end{array}
        \right.
    \end{align*}
    Since $\sigma_{n}^{2/3}\sqrt{|\left(x_{\alpha}^{\delta}\right)_{n}|}=\eta_{n}>0$, both $\eta_{n,2}$ and $\eta_{n,3}$ are solutions to \eqref{equ3.26}. Furthermore, we can verify that $\eta_{n,3}>\eta_{n,2}$, establishing $\eta_{n,3}$ as the unique solution of \eqref{equ3.25}, which is given by
    \begin{equation*}
        \left(x_{\alpha}^{\delta}\right)_{n}=\frac{2}{3}\frac{y^{\delta}_{n}}{\sigma_{n}}\left(1+cos\left(\frac{2}{3}\pi-\frac{2}{3}arccos\left(\frac{\alpha}{8}\left(\frac{\sigma_{n}^{1/3}y^{\delta}_{n}}{3}\right)^{-3/2}\right)\right)\right).
    \end{equation*}
    
    \par For the scenario where $\sigma_{n}^{1/3}y^{\delta}_{n}< -\frac{3}{4}\alpha^{2/3}$, we define $\sigma_{n}^{2/3}\sqrt{|\left(x_{\alpha}^{\delta}\right)_{n}|}=\eta_{n}$, leading to $\sigma_{n}^{4/3}\left(x_{\alpha}^{\delta}\right)_{n}=-\eta_{n}^{2}$. Thus, equation \eqref{equ3.25} can be transformed into the cubic equation
    \begin{equation}\label{equ3.27}
        \eta_{n}^{3}+\sigma_{n}^{1/3}y^{\delta}_{n}\eta_{n}+\frac{\alpha}{4}=0.
    \end{equation}
    Following a similar analysis as before, we conclude that \eqref{equ3.25} has a unique solution given by
    \begin{equation*}
        \left(x_{\alpha}^{\delta}\right)_{n}=-\frac{2}{3}\left|\frac{y^{\delta}_{n}}{\sigma_{n}}\right|\left(1+cos\left(\frac{2}{3}\pi-\frac{2}{3}arccos\left(\frac{\alpha}{8}\left(\frac{\sigma_{n}^{1/3}y^{\delta}_{n}}{3}\right)^{-3/2}\right)\right)\right).
    \end{equation*}
    
    \par For the case where $\left|\sigma_{n}^{1/3}y^{\delta}_{n}\right|\leq \frac{3}{4}\alpha^{2/3}$, we define $\left(x_{\alpha}^{\delta}\right)_{n}=0$. In summary, we demonstrate that
    \begin{align*}
        \left(x_{\alpha}^{\delta}\right)_{n}=
        \left\{             
             \begin{array}{ll}
             \frac{2}{3}\frac{y^{\delta}_{n}}{\sigma_{n}}\left(1+cos\left(\frac{2}{3}\pi-\frac{2}{3}arccos\left(\frac{\alpha}{8}\left(\frac{\sigma_{n}^{1/3}y^{\delta}_{n}}{3}\right)^{-3/2}\right)\right)\right), & \sigma_{n}^{1/3}y^{\delta}_{n}> \frac{3}{4}\alpha^{2/3}, \\[10pt]
             0, & \left|\sigma_{n}^{1/3}y^{\delta}_{n}\right|\leq \frac{3}{4}\alpha^{2/3} , \\[6pt]
             -\frac{2}{3}\left|\frac{y^{\delta}_{n}}{\sigma_{n}}\right|\left(1+cos\left(\frac{2}{3}\pi-\frac{2}{3}arccos\left(\frac{\alpha}{8}\left(\frac{\sigma_{n}^{1/3}y^{\delta}_{n}}{3}\right)^{-3/2}\right)\right)\right), & \sigma_{n}^{1/3}y^{\delta}_{n}< -\frac{3}{4}\alpha^{2/3}.
             \end{array}
        \right.
    \end{align*}
    Thus, we conclude that
    \begin{equation*}
         x_{\alpha}^{\delta}=\sum_{n\in\mathbb{N}}\mathcal{H}_{\alpha,n} \left(\sigma_{n}^{1/3}y^{\delta}_{n}\right)v_{n},
    \end{equation*}
    where $\mathcal{H}_{\alpha,n}$ is defined in Definition \ref{def3.10}.
    
    \par The remainder of this section will  focus on proving that the solution of the SVD algorithm for the sparsity regularization with $\ell^{1/2}$ penalties is a stationary point of the function $\Phi_{\alpha}\left(x\right)$ as defined in \eqref{equ3.21}. Let us define the set $\Lambda=\{n\in\mathbb{N} \mid \left(x_{\alpha}^{\delta}\right)_{n}\neq 0\}$. For $x_{\alpha}^{\delta}+h\in U\left(x_{\alpha}^{\delta},\rho\right)$ with $0<\rho<\sum_{n\notin\Lambda}\alpha^{2} /(4\sigma_{n}^{2} {y^{\delta}_{n}}^{2})$, we can derive the following
    \begin{align}\label{equ3.28}
        \Phi_{\alpha}\left(x_{\alpha}^{\delta}+h\right)
        & =\left\|K\left(x_{\alpha}^{\delta}+h\right)-y^{\delta}\right\|_{Y}^{2}+ \alpha\sum_{n\in\mathbb{N}}\sqrt{\left|\left(x_{\alpha}^{\delta}\right)_{n} +h_{n}\right|} \nonumber\\
        & =\left\|Kx_{\alpha}^{\delta}-y^{\delta}\right\|_{Y}^{2}+ 2\left<Kx_{\alpha}^{\delta}-y^{\delta},Kh\right>+\left\|Kh\right\|_{Y}^{2}+ \alpha\sum_{n\in\mathbb{N}}\sqrt{\left|\left(x_{\alpha}^{\delta}\right)_{n}+ h_{n}\right|} \nonumber\\
        & =\Phi_{\alpha}\left(x_{\alpha}^{\delta}\right)+2\left<Kx_{\alpha}^{\delta}-y^{\delta},Kh\right> +\left\|Kh\right\|_{Y}^{2}+\alpha\sum_{n\in\mathbb{N}}\sqrt{\left|\left(x_{\alpha}^{\delta}\right)_{n} +h_{n}\right|}-\alpha\sum_{n\in\mathbb{N}}\sqrt{\left|\left(x_{\alpha}^{\delta}\right)_{n}\right|}.
    \end{align}
    Since $x_{\alpha}^{\delta}$ can be expressed as shown in \eqref{equ3.24}, it fulfills the first-order optimal condition \eqref{equ3.25}. Therefore, we have
    \begin{equation*}
        2\left<Kx_{\alpha}^{\delta}-y^{\delta},Kh\right>= \sum_{n\in\mathbb{N}}2\sigma_{n}\left(\sigma_{n}\left(x_{\alpha}^{\delta}\right)_{n} -y^{\delta}_{n}\right) h_{n} =\sum_{n\in\mathbb{N}}-\frac{\alpha {\rm sign} (\left(x_{\alpha}^{\delta}\right)_{n})h_{n}} {2\sqrt{|\left(x_{\alpha}^{\delta}\right)_{n}|}}.
    \end{equation*}
    Next, let us define $\Lambda:=\left\{n\in\mathbb{N} \mid \left(x_{\alpha}^{\delta}\right)_{n}\neq 0\right\}$ and rewrite \eqref{equ3.28} as
    \begin{align*}
        \Phi_{\alpha}\left(x_{\alpha}^{\delta}+h\right)-\Phi_{\alpha}\left(x_{\alpha}^{\delta}\right)
        & =2\sum_{n\in\mathbb{N}}\sigma_{n}\left(\sigma_{n}\left(x_{\alpha}^{\delta}\right)_{n} -y^{\delta}_{n}\right)h_{n}+\left\|Kh\right\|^{2} \nonumber\\
        & \quad +\alpha\sum_{n\in\mathbb{N}} \sqrt{\left|\left(x_{\alpha}^{\delta}\right)_{n}+h_{n}\right|} -\alpha\sum_{n\in\mathbb{N}}\sqrt{\left|\left(x_{\alpha}^{\delta}\right)_{n}\right|} \nonumber\\
        & =\left\|Kh\right\|^{2}+\sum_{n\notin\Lambda}\left[\alpha\sqrt{h_{n}}-2\sigma_{n}y_{n}h_{n}\right] \nonumber\\
        & \quad+\sum_{n\in\Lambda}\alpha\left[\sqrt{\left|\left(x_{\alpha}^{\delta}\right)_{n}+h_{n}\right|} -\sqrt{\left|\left(x_{\alpha}^{\delta}\right)_{n}\right|}-\frac{{\rm sign}\left(\left(x_{\alpha}^{\delta}\right)_{n}\right)h_{n} }{2\sqrt{\left|\left(x_{\alpha}^{\delta}\right)_{n} \right|}}\right].
    \end{align*}
    For $n\notin\Lambda$, since $|h_{n}|\leq|\rho|\leq\alpha^{2}/(4\sigma_{n}^{2} {y^{\delta}_{n}}^{2})$, we find
    \begin{equation*}
        \alpha\sqrt{h_{n}}-2\sigma_{n}y_{n}h_{n}\geq 0.
    \end{equation*}
    For $n\in\Lambda$, we explore two scenarios. Case 1, when $\left|\left(x_{\alpha}^{\delta}\right)_{n}+h_{n}\right|\leq \left|\left(x_{\alpha}^{\delta}\right)_{n}\right|$, we analyze 
    \begin{align*}
        & \quad \sqrt{\left|\left(x_{\alpha}^{\delta}\right)_{n}+h_{n}\right|}- \sqrt{\left|\left(x_{\alpha}^{\delta}\right)_{n}\right|}- \frac{h_{n}{\rm sign} \left(\left(x_{\alpha}^{\delta}\right)_{n}\right)} {2\sqrt{\left|\left(x_{\alpha}^{\delta}\right)_{n}\right|}} \nonumber\\
        & =\frac{1}{\sqrt{\left|\left(x_{\alpha}^{\delta}\right)_{n}+h_{n}\right|}} \left[\left|\left(x_{\alpha}^{\delta}\right)_{n}+h_{n}\right|- \sqrt{\left|\left(x_{\alpha}^{\delta}\right)_{n}\right|} \sqrt{\left|\left(x_{\alpha}^{\delta}\right)_{n}+h_{n}\right|}\right]\nonumber\\
        & \quad -\frac{1}{\sqrt{\left|\left(x_{\alpha}^{\delta}\right)_{n}+h_{n}\right|}} \left[\frac{\sqrt{\left|\left(x_{\alpha}^{\delta}\right)_{n}+h_{n}\right|}} {2\sqrt{\left|\left(x_{\alpha}^{\delta}\right)_{n}\right|}} h_{n} {\rm sign} \left(\left(x_{\alpha}^{\delta}\right)_{n}\right)\right] \nonumber\\
        & \geq \frac{1}{\sqrt{\left|\left(x_{\alpha}^{\delta}\right)_{n}+h_{n}\right|}} \left[\left|\left(x_{\alpha}^{\delta}\right)_{n}+h_{n}\right|- \frac{1}{2}\left|\left(x_{\alpha}^{\delta}\right)_{n}\right|- \frac{1}{2}\left|\left(x_{\alpha}^{\delta}\right)_{n}+h_{n}\right|- \frac{1}{2}h_{n}{\rm sign} \left(x_{\alpha}^{\delta}\right)_{n}\right].
    \end{align*}
    If $\left(x_{\alpha}^{\delta}\right)_{n}>0$, we observe
    \begin{equation*}
        \frac{1}{2}\left|\left(x_{\alpha}^{\delta}\right)_{n}+h_{n}\right|- \frac{1}{2}\left(\left(x_{\alpha}^{\delta}\right)_{n}+h_{n}\right)\geq 0.
    \end{equation*}
    If $\left(x_{\alpha}^{\delta}\right)_{n}<0$, we see
    \begin{equation*}
        \frac{1}{2}\left|\left(x_{\alpha}^{\delta}\right)_{n}+h_{n}\right|+ \frac{1}{2}\left(\left(x_{\alpha}^{\delta}\right)_{n}+h_{n}\right)\geq 0.
    \end{equation*}
    Thus, we conclude that
    \begin{equation*}
        \sqrt{\left|\left(x_{\alpha}^{\delta}\right)_{n}+h_{n}\right|}- \sqrt{\left|\left(x_{\alpha}^{\delta}\right)_{n}\right|}- \frac{{\rm sign} \left(\left(x_{\alpha}^{\delta}\right)_{n}\right)} {2\sqrt{\left|\left(x_{\alpha}^{\delta}\right)_{n}\right|}}\geq 0.
    \end{equation*}
    In this case, we have
    \begin{equation*}
        \Phi_{\alpha}\left(x_{\alpha}^{\delta}+h\right)- \Phi_{\alpha}\left(x_{\alpha}^{\delta}\right)\geq \left\|Kh\right\|^{2}.
    \end{equation*}
    Case 2, when $\left|\left(x_{\alpha}^{\delta}\right)_{n}+h_{n}\right| > \left|\left(x_{\alpha}^{\delta}\right)_{n}\right|$, we analyze
    \begin{align*}
        \Phi_{\alpha}\left(x_{\alpha}^{\delta}+h\right)-\Phi_{\alpha}\left(x_{\alpha}^{\delta}\right)
        & =\sum_{n\notin\Lambda}\left[\sigma_{n}^{2}h_{n}^{2}+\alpha\sqrt{h_{n}}-2\sigma_{n}y_{n}h_{n}\right] \\
        & \quad +\sum_{n\in\Lambda}\left[\sigma_{n}^{2}h_{n}^{2}+ \alpha\sqrt{\left|\left(x_{\alpha}^{\delta}\right)_{n}+h_{n}\right|}- \alpha\sqrt{\left|\left(x_{\alpha}^{\delta}\right)_{n}\right|}- \frac{\alpha h_{n}{\rm sign}\left(\left(x_{\alpha}^{\delta}\right)_{n}\right)} {2\sqrt{\left|\left(x_{\alpha}^{\delta}\right)_{n}\right|}}\right].
    \end{align*}
    The latter part of the inequality can be simplified to
    \begin{align*}
        & \quad \sqrt{\left|\left(x_{\alpha}^{\delta}\right)_{n}+h_{n}\right|}- \sqrt{\left|\left(x_{\alpha}^{\delta}\right)_{n}\right|}- \frac{h_{n}{\rm sign} \left(\left(x_{\alpha}^{\delta}\right)_{n}\right)} {2\sqrt{\left|\left(x_{\alpha}^{\delta}\right)_{n}\right|}} \\
        & =\frac{1}{\sqrt{\left|\left(x_{\alpha}^{\delta}\right)_{n}+h_{n}\right|}+ \sqrt{\left|\left(x_{\alpha}^{\delta}\right)_{n}\right|}} \left[\left|\left(x_{\alpha}^{\delta}\right)_{n}+h_{n}\right|- \left|\left(x_{\alpha}^{\delta}\right)_{n}\right|-h_{n}{\rm sign} \left(\left(x_{\alpha}^{\delta}\right)_{n}\right)\right] \\
        & \quad -\frac{h_{n}{\rm sign} \left(\left(x_{\alpha}^{\delta}\right)_{n}\right) \left[\sqrt{\left|\left(x_{\alpha}^{\delta}\right)_{n}+h_{n}\right|}- \sqrt{\left|\left(x_{\alpha}^{\delta}\right)_{n}\right|}\right]} {2\sqrt{\left|\left(x_{\alpha}^{\delta}\right)_{n}\right|} \left[\sqrt{\left|\left(x_{\alpha}^{\delta}\right)_{n}+h_{n}\right|}+ \sqrt{\left|\left(x_{\alpha}^{\delta}\right)_{n}\right|}\right]} \\
        & \geq  -\frac{h_{n} {\rm sign} \left(\left(x_{\alpha}^{\delta}\right)_{n}\right) \left[\left|\left(x_{\alpha}^{\delta}\right)_{n}+h_{n}\right|- \left|\left(x_{\alpha}^{\delta}\right)_{n}\right|\right]} {2\sqrt{\left|\left(x_{\alpha}^{\delta}\right)_{n}\right|} \left[\sqrt{\left|\left(x_{\alpha}^{\delta}\right)_{n}+h_{n}\right|}+ \sqrt{\left|\left(x_{\alpha}^{\delta}\right)_{n}\right|}\right]^{2}} \nonumber\\
        & \geq -\frac{h_{n}^{2}}{8\left|\left(x_{\alpha}^{\delta}\right)_{n}\right|^{3/2}}.
    \end{align*}
    Combine the above inequality gives us
    \begin{align*}
        \Phi_{\alpha}\left(x_{\alpha}^{\delta}+h\right)-\Phi_{\alpha}\left(x_{\alpha}^{\delta}\right)
        & \geq\sum_{n\notin\Lambda}\left[\sigma_{n}^{2}h_{n}^{2}+\frac{\alpha}{2}\sqrt{h_{n}} -2\sigma_{n}y_{n}h_{n}\right] +\sum_{n\in\Lambda}\left[\sigma_{n}^{2}h_{n}^{2} -\frac{h_{n}^{2}}{8\left|\left(x_{\alpha}^{\delta}\right)_{n}\right|^{3/2}}\right].
    \end{align*}
    Given $\left|\left<x_{\alpha}^{\delta},v_{n}\right>\right|^{2}= \left\|x_{\alpha}^{\delta}\right\|_{\ell_{2}}^{2}>C^{2}$ and $c(\alpha)\sigma_{n}>1$,
    \begin{equation*}
        \sigma_{n}^{2}h_{n}^{2}-\frac{h_{n}^{2}}{8\left|\left(x_{\alpha}^{\delta}\right)_{n}\right| ^{3/2}} \geq \sigma_{n}^{2}h_{n}^{2}-\frac{h_{n}^{2}}{8|C|^{3/2}} = \sigma_{n}^{2}h_{n}^{2}-\frac{h_{n}^{2}}{c(\alpha)^{2}}\geq 0,
    \end{equation*}
    where $c(\alpha)=2\sqrt{2}C^{3/4}$. In this case, we find 
    \begin{equation*}
        \Phi_{\alpha}\left(x_{\alpha}^{\delta}+h\right)-\Phi_{\alpha}\left(x_{\alpha}^{\delta}\right)\geq\sum_{n\notin\Lambda}\sigma_{n}^{2}h_{n}^{2}.
    \end{equation*}
    \par In summary, $x_{\alpha}^{\delta}$ serves as a stationary point of the functional $\Phi_{\alpha}\left(x\right)$ as defined in \eqref{equ3.21}, and it represents a local minimum solution within $U(x^{*},\rho)$ for $0<\rho<\sum_{n\notin\Lambda}\alpha^{2} /(4\sigma_{n}^{2} {y^{\delta}_{n}}^{2})$.
\end{proof}

\par Although Theorem \ref{the3.12} specifies that the operator $K$ is diagonal under a specific orthogonal bases, it suggests a new SVD method for a more general $K$. As $\alpha\rightarrow 0$, the nonlinear half thresholding function $\mathcal{H}_{\alpha}$ in \eqref{equ3.24} degenerates into an identity operator and \eqref{equ3.24} degenerates to \eqref{equ2.4}. Therefore, we further prove \eqref{equ3.24} is a regularization strategy, whether the operator $K$ is diagonal or not.

\begin{theorem}\label{the3.13}
    Let $K: X\rightarrow Y$ be a linear compact operator with a singular system denoted by $\left\{\left(\sigma_{n},v_{n},u_{n}\right)\right\}_{n\in\mathbb{N}}$. For every $\alpha>0$ and $0<\sigma_{n}\leq\|K\|$, suppose there exists $c(\alpha)$ such that
    \begin{equation*}
        c(\alpha)\sigma_{n}\geq \frac{4}{3},
    \end{equation*} 
    where $\sigma_{n}$ satisfies $\left|\sigma_{n}^{1/3}\left<y^{\delta},u_{n}\right>\right|>\frac{3}{4}\alpha^{2/3}$. Then the operator $\mathcal{R}_{\alpha}:Y\rightarrow \ell_{2}$, $\alpha>0$, defined by
    \begin{equation*}
        \mathcal{R}_{\alpha}y^{\delta}:=\sum_{n\in\mathbb{N}}\mathcal{H}_{\alpha,n}(\sigma_{n}^{1/3}\left<y^{\delta}, u_{n}\right>)v_{n},
        \ y^{\delta}\in Y
    \end{equation*}
     where $\mathcal{H}_{\alpha,n}$ is the half thresholding function described in  Definition \ref{def3.10}, constitutes a regularization strategy with $\|\mathcal{R}_{\alpha}\|\leq c(\alpha)$.
\end{theorem}

\begin{proof}
    We begin by establishing the boundedness of the operator $\mathcal{R}_{\alpha}$. We define the sets $\mathbb{N}_{1}=\left\{n\in\mathbb{N}\mid\left|\sigma_{n}^{1/3}\left<y^{\delta},u_{n}\right>\right| > \frac{3}{4}\alpha^{2/3}\right\}$ and $\mathbb{N}_{2}=\left\{n\in\mathbb{N}\mid \left|\sigma_{n}^{1/3}\left<y^{\delta},u_{n}\right>\right| \leq \frac{3}{4}\alpha^{2/3}\right\}$. In this case, it is evident that $\mathbb{N}=\mathbb{N}_{1}\cup\mathbb{N}_{2}$. We have  
    \begin{equation*}
        \|\mathcal{R}_{\alpha}y^{\delta}\|^2=\sum_{n\in\mathbb{N}_{1}\cup\mathbb{N}_{2}} \left[\mathcal{H}_{\alpha,n}\left(\sigma_{n}^{1/3}\left<y^{\delta}, u_{n}\right>\right)\right]^2.
    \end{equation*}
    For $n\in\mathbb{N}_{1}$, we have
    \begin{align}\label{equ3.29}
        \sum_{n\in\mathbb{N}_{1}}\left[\mathcal{H}_{\alpha,n}\left(\sigma_{n}^{1/3}\left<y^{\delta}, u_{n}\right>\right)\right]^2
        & =\sum_{n\in\mathbb{N}_{1}}\left[\frac{2 \left<y^{\delta}, u_{n}\right>}{3\sigma_{n}}\left(1+cos\left(\frac{2}{3}\pi-\frac{2}{3}\phi_{\alpha}\left( \sigma_{n}^{1/3}\left<y^{\delta}, u_{n}\right>\right)\right)\right)\right]^2 \nonumber\\
        & \leq \frac{16}{9}\sum_{n\in\mathbb{N}_{1}} \frac{1}{\sigma_{n}^{2}}|\left<y^{\delta},u_{n}\right>|^2.
    \end{align}
    For $n\in\mathbb{N}_{2}$, we find
    \begin{align}\label{equ3.30}
        \sum_{n\in\mathbb{N}_{2}}\left[\mathcal{H}_{\alpha,n}\left(\sigma_{n}^{1/3}\left<y^{\delta}, u_{n}\right>\right)\right]^2
        & =0.
    \end{align}
    By combining \eqref{equ3.29} and \eqref{equ3.30}, we derive
    \begin{align*}
        \|\mathcal{R}_{\alpha}y^{\delta}\|^{2}
        & = \sum_{n\in\mathbb{N}_{1}\cup\mathbb{N}_{2}}\left[\mathcal{H}_{\alpha,n}\left(\sigma_{n}^{1/3}\left<y^{\delta}, u_{n}\right>\right)v_{n}\right]^2 \\ 
        & =  \sum_{n\in\mathbb{N}_{1}}\left[\mathcal{H}_{\alpha,n}\left(\sigma_{n}^{1/3}\left<y^{\delta}, u_{n}\right>\right)v_{n}\right]^2+\sum_{n\in\mathbb{N}_{2}}\left[\mathcal{H}_{\alpha,n}\left(\sigma_{n}^{1/3}\left<y^{\delta}, u_{n}\right>\right)v_{n}\right]^2 \\
        & \leq \frac{16}{9}\sum_{n\in\mathbb{N}_{1}} \frac{1}{\sigma_{n}^{2}}|\left<y^{\delta},u_{n}\right>|^2+0 \\
        & \leq c(\alpha)^{2}\|y^{\delta}\|^{2}.
    \end{align*}
    Thus, we conclude that $\|\mathcal{R}_{\alpha}\|\leq c(\alpha)$.
    
    \par Next, we aim to demonstrate the convergence of $\mathcal{R}_{\alpha}Kx$ to $x$ as $\alpha\rightarrow 0$. By invoking Lemma \ref{lem2.6}, we have $\sigma_{n}^{1/3}\left<y,u_{n}\right>=\sigma_{n}^{4/3}\left<x,v_{n}\right>$.
     For $n\in\mathbb{N}_{1}$, where $\left|\sigma_{n}^{4/3}\left<x,v_{n}\right>\right| >\frac{3}{4}\alpha^{2/3}$, we have
    \begin{align}\label{equ3.31}
        & \quad \sum_{n\in\mathbb{N}_{1}}\left\|\mathcal{H}_{\alpha,n}\left(\sigma_{n}^{4/3}\left<x, v_{n}\right>\right)v_{n}-\left<x,v_{n}\right>v_{n}\right\|^2 \nonumber\\
        & = \sum_{n\in\mathbb{N}_{1}}\left\|\frac{2 \sigma_{n} \left<x,v_{n}\right>}{3\sigma_{n}}\left(1+cos\left(\frac{2}{3}\pi-\frac{2}{3}\phi_{\alpha}\left( \sigma_{n}^{4/3}\left<x,v_{n}\right>\right)\right)\right)v_{n} -\left<x,v_{n}\right>v_{n}\right\|^{2} \nonumber\\
        & = \sum_{n\in\mathbb{N}_{1}}\left\|\left(\frac{2}{3}cos\left(\frac{2}{3}\pi-\frac{2}{3}\phi_{\alpha}\left( \sigma_{n}^{4/3}\left<x,v_{n}\right>\right)\right)-\frac{1}{3}\right)\left<x, v_{n}\right>v_{n}\right\|^{2}
    \end{align}
    For $n\in\mathbb{N}_{2}$, where $\left|\sigma_{n}^{4/3}\left<x,v_{n}\right>\right| \leq\frac{3}{4}\alpha^{2/3}$, we find
    \begin{align}\label{equ3.32}
        \sum_{n\in\mathbb{N}_{2}}\left\|\mathcal{H}_{\alpha,n}\left(\sigma_{n}^{4/3}\left<x,v_{n}\right>\right)v_{n}-\left<x,v_{n}\right>v_{n}\right\|^2
        & = \sum_{n\in\mathbb{N}_{2}}\left\|\left<x,v_{n}\right>v_{n}\right\|^2 \nonumber\\
        & \leq \sum_{\{n\in\mathbb{N}_{2}\}\cap\{n\leq N\}}\frac{3 \alpha^{2/3}}{4\sigma_{n}^{4/3}}+\sum_{\{n\in\mathbb{N}_{2}\}\cap \{n> N\}}|\left<x,v_{n}\right>|^{2}.
    \end{align}
     \par For every $\epsilon>0$, there exist an $N\in\mathbb{N}$ such that for any fixed $x\in X$,
    \begin{equation*}
        \sum_{n=N+1}^{\infty}|\left<x,v_{n}\right>|^{2}\leq \frac{\epsilon^{2}}{2}.
    \end{equation*}
    As $\alpha\rightarrow 0$, for all $i=1,\dots,N$, we derive the following results
    \begin{equation*}
        |\alpha-0|^{2}\leq \frac{8\sigma_{i}^{4}}{27N}\epsilon^{6} \leq\frac{8\left\|K\right\|^{4}}{27N}\epsilon^{6}
    \end{equation*}
    and
    \begin{equation*}
        \lim_{\alpha\rightarrow 0}cos\left(\frac{2}{3}\pi-\frac{2}{3}\phi_{\alpha}\left( \sigma_{n}^{4/3}\left<x,v_{n}\right>\right)\right)=\frac{1}{2}.
    \end{equation*}
    By combining equations \eqref{equ3.31} and \eqref{equ3.32}, we obtain the following result
    \begin{align*}
        \|\mathcal{R}_{\alpha}Kx-x\|^2
        & = \sum_{n\in\mathbb{N}_{1}\cup\mathbb{N}_{2}} \left\|\mathcal{H}_{\alpha,n}\left(\sigma_{n}^{4/3}\left<x,v_{n}\right>\right)v_{n}-\left<x,v_{n}\right>v_{n}\right\|^2 \nonumber\\
        & \leq 0+\sum_{\{n\in\mathbb{N}_{2}\}\cap\{n\leq N\}}\frac{3\alpha^{2/3}}{4\sigma^{4/3}}+\sum_{\{n\in\mathbb{N}_{2}\}\cap \{n> N\}}|\left<x,v_{n}\right>|^{2} \\
        & \leq \sum_{n\leq N}\frac{3}{4\sigma^{4/3}}\left(\frac{8\sigma_{i}^{4}}{27N}\epsilon^{6}\right)^{1/3}+\sum_{n> N}|\left<x,v_{n}\right>|^{2} \nonumber\\
        & \leq \frac{\epsilon^{2}}{2}+\frac{\epsilon^{2}}{2}= \epsilon^{2}.
    \end{align*}
    Thus, since $\{\sigma_{n}\}_{n\in\mathbb{N}}$ represents the ordered sequence of the positive singular values of $K$, we have demonstrated that for every $x\in X$,
    \begin{equation*}
        \mathcal{R}_{\alpha}Kx\rightarrow x\ (\alpha\rightarrow 0).
    \end{equation*}
    This establishes the proof of convergence.
\end{proof}

\par The SVD method that utilizes half thresholding function is presented in Algorithm \ref{alg2}.

\begin{algorithm}
 \caption{SVD algorithm for the $\ell^{1/2}$-SVD regularization strategy}
\begin{algorithmic}\label{alg2}
\STATE{\textbf{Input:} Given regularization parameter $\alpha>0$, 
observation data $y^{\delta}$ and linear operator $K$.}
\STATE{\qquad $[U^{T},\Sigma,V]$ = SVD($K$)}
\STATE{\qquad Let $U^{T}=[u_{1},u_{2},\cdots,u_{m}]$, $V^{T}=[v_{1},v_{2},\cdots,v_{n}]$, and $\sigma_{i}=\Sigma(i,i)$.}
\STATE{\qquad For $i=1\cdots \min\{m,n\}$ :}
\STATE{\qquad\qquad $x_{i}^{\delta}=\mathcal{H}_{\alpha,i}\left(\sigma_{i}^{1/3}\left<y^{\delta}, u_{i}\right>\right)v_{i}$}
\STATE{\qquad end}
\STATE{\qquad $x_{\alpha}^{\delta}=\sum_{i=1}^{\min\{m,n\}}(x_{i}^{\delta})$}
\STATE{\textbf{Output:} The regularized solution $x_{\alpha}^{\delta}$.}
\end{algorithmic}
\end{algorithm}

\section{Numerical experiments}\label{sec4}

In this section, we present results from two numerical examples to illustrate the effectiveness of the proposed algorithm. We provide comparisons among the SVD algorithms, ISTA as described in \cite{DFL08}, FISTA from \cite{BT09}, and the PG algorithm outlined in \cite{DH20}. The relative error is employed to assess the performance of the reconstruction $x_{\alpha}^{\delta}$:
\begin{equation*}
    {\rm Rerror} = \frac{\|x_{\alpha}^{\delta}-x^{true}\|_{\ell_{2}}}{\|x^{true}\|_{\ell_{2}}},
\end{equation*}
where $x_{\alpha}^{\delta}$ is defined in Definition \ref{def2.1} and $x^{true}$ represents the exact solution. Additionally, we analyze the running time of the algorithms and the probability of successful recovery. The first example addresses a well-conditioned compressive sensing problem, while the second example pertains to an ill-conditioned image deblurring issue. All numerical experiments were conducted using MATLAB R2024a on a workstation equipped with an AMD Ryzen 5 4500U processor with Radeon Graphics operating at 2.38 GHz and 16GB of RAM.

\subsection{Compressive sensing}\label{sec4.1}

In the first example, we examine compressive sensing using the widely utilized random Gaussian matrix. The compressive sensing problem is defined as follows
\begin{equation*}
    K_{m \times n}x_{n}=y_{m},
\end{equation*}
where $K_{m \times n}$ represents a well-conditioned random Gaussian matrix generated by $K={\rm randn}(m,n)$. The exact data $y^{\dag}$ is obtained by $y^{\dag}=Kx^{\dag}$, with the exact solution $x^{\dag}$ being an $s$-sparse signal supported by a random index set. White Gaussian noise is added to the exact data $y^{\dag}$ by calling $y^{\delta}={\rm awgn}(Kx^{\dag},\sigma)$, where $\sigma$ (measured in dB) measures the ratio between the true (noise free) data $y^{\dag}$ and Gaussian noise. A larger value of $\sigma$ corresponds to a smaller noise level $\delta$, which is defined as $\delta=\|y^{\delta}-y^{\dag}\|_{\ell_{2}}$. For the compressive sensing problem, if $\|KK\|_{\ell_{2}} > 1$, it is necessary to scale the matrix $K$ to adapt to ISTA, FISTA, and PG iterative algorithms, that is, $K_{m\times n}\rightarrow cK_{m\times n}$ \cite{DDD04}. It is important to note that the condition number remains unchanged when the matrix is rescaled.

\begin{figure}[htbp]
    \centering
    \subfigure[True signal.]{\includegraphics[scale=0.26]{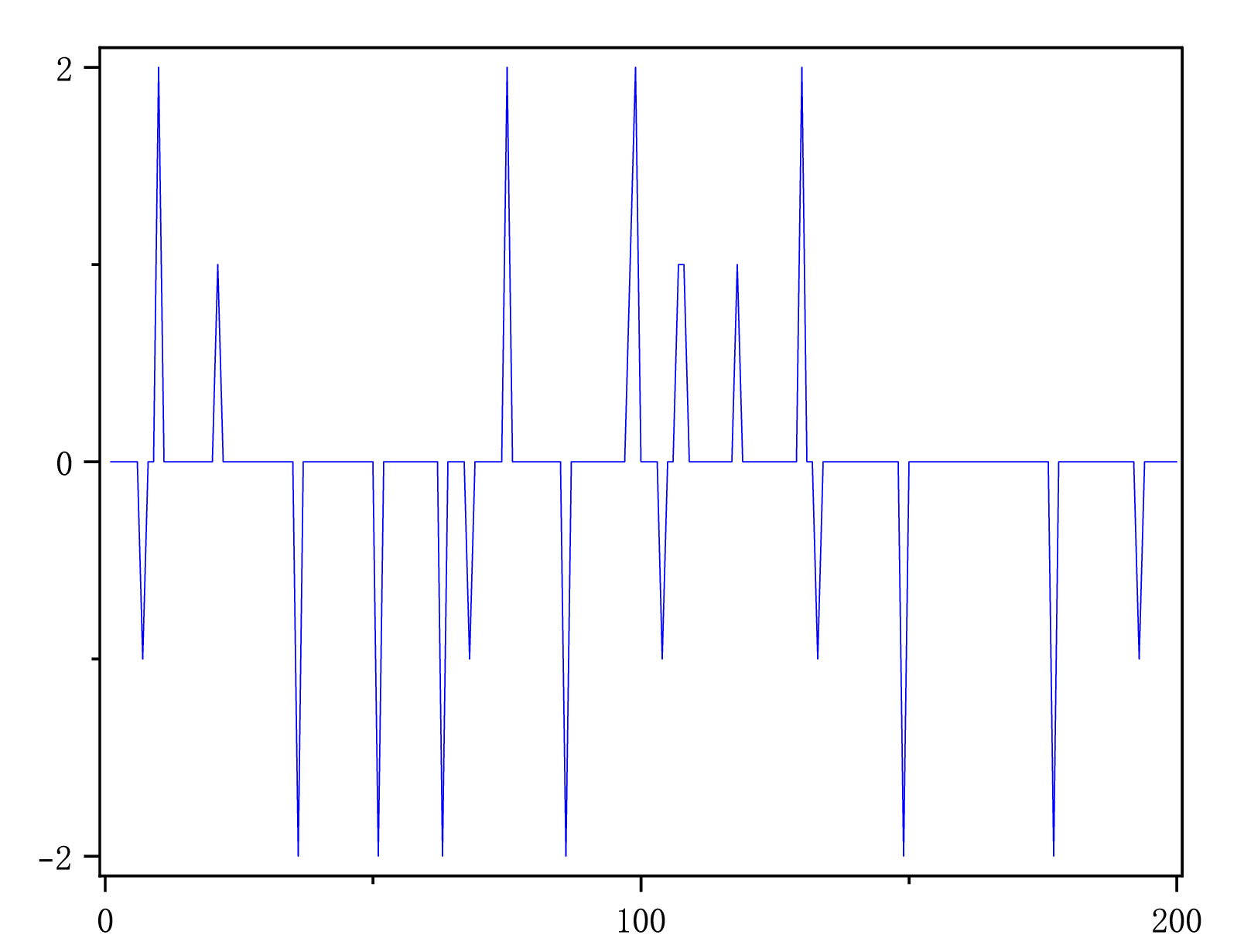}} 
    \subfigure[Observed data ($\sigma=80$dB).]{\includegraphics[scale=0.26]{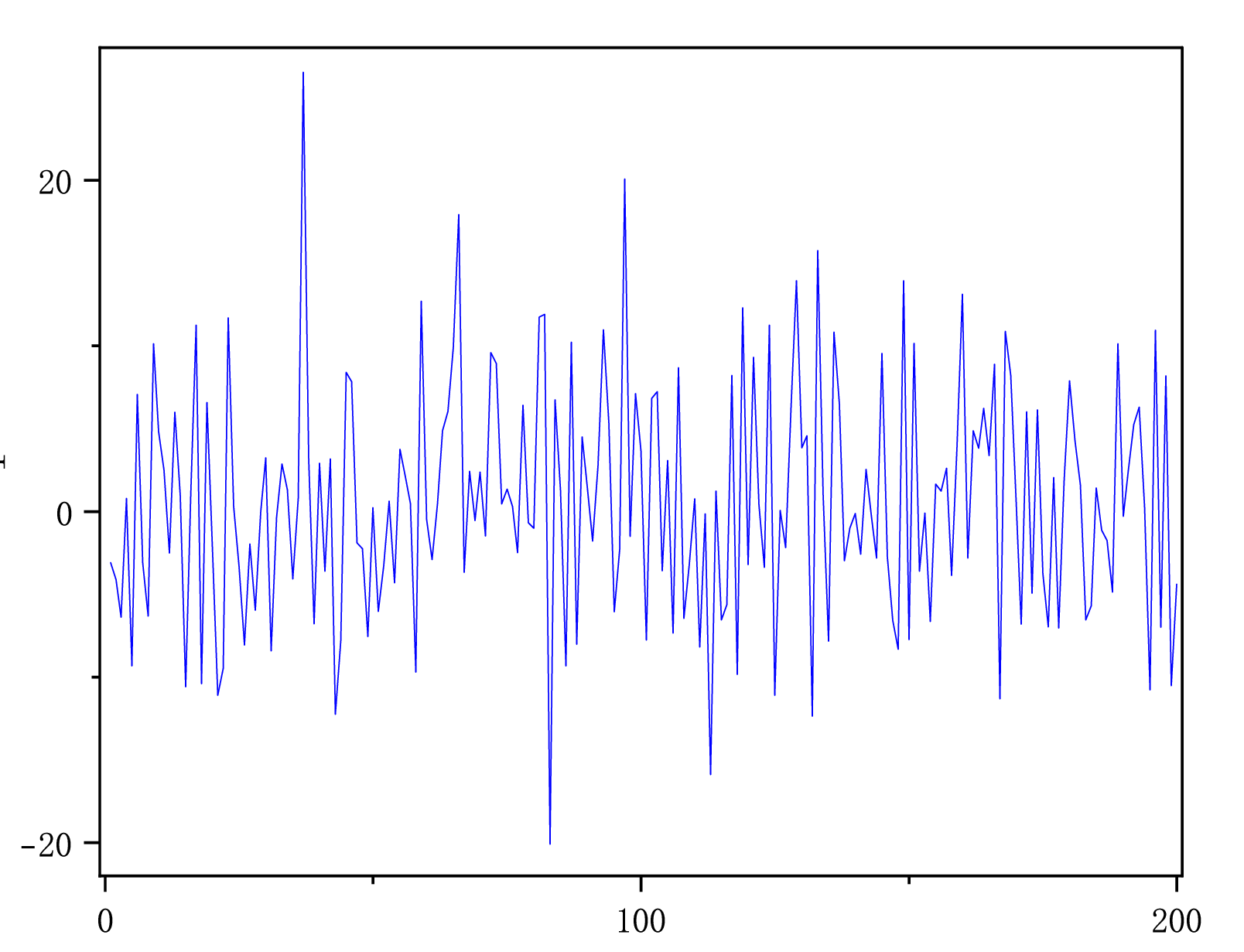}}
    \caption{(a) is the true signal and (b) is the observed data with $80$dB noise.}
    \label{figure:1}
\end{figure}

\begin{figure}[htbp]
    \centering
    \subfigure[True signal.]{\includegraphics[scale=0.26]{real_signal.eps}}
    \subfigure[Recovery signal by ISTA ({\rm Rerror}= 0.0078).]{\includegraphics[scale=0.26]{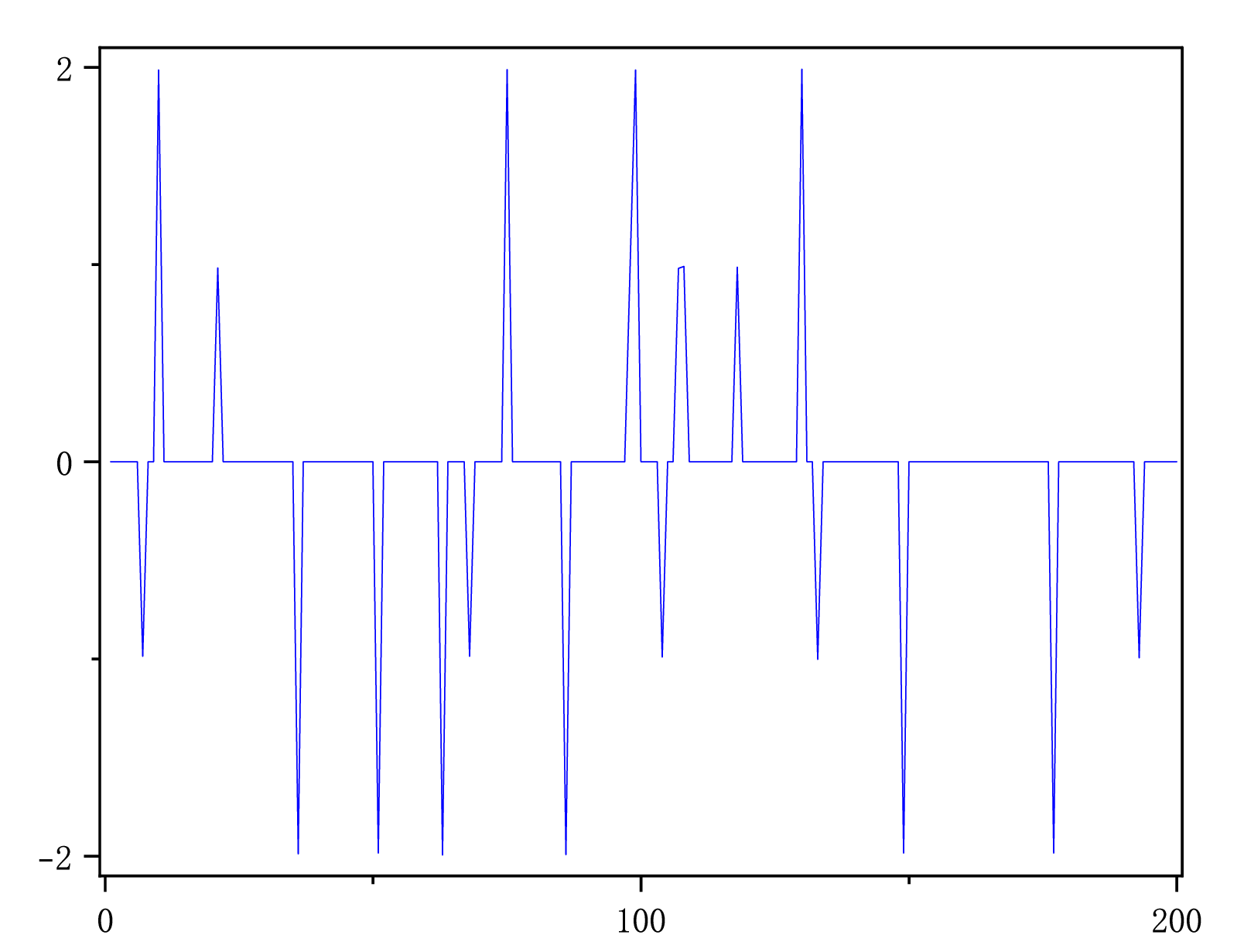}}\\
    \subfigure[Recovery signal by FISTA ({\rm Rerror}= 0.0064).]{\includegraphics[scale=0.26]{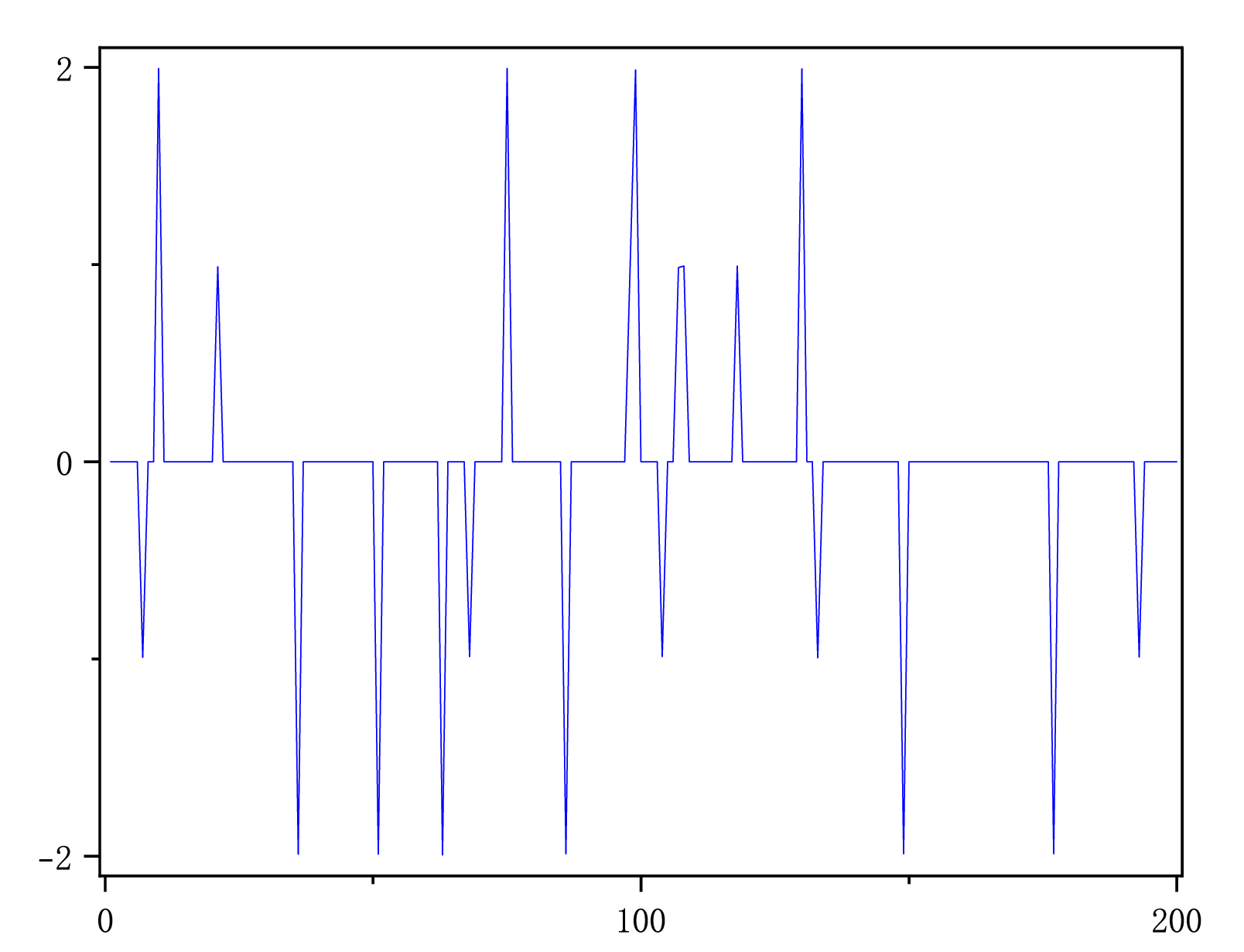}}
    \subfigure[Recovery signal by PG algorithm ({\rm Rerror}= 0.0029).]{\includegraphics[scale=0.26]{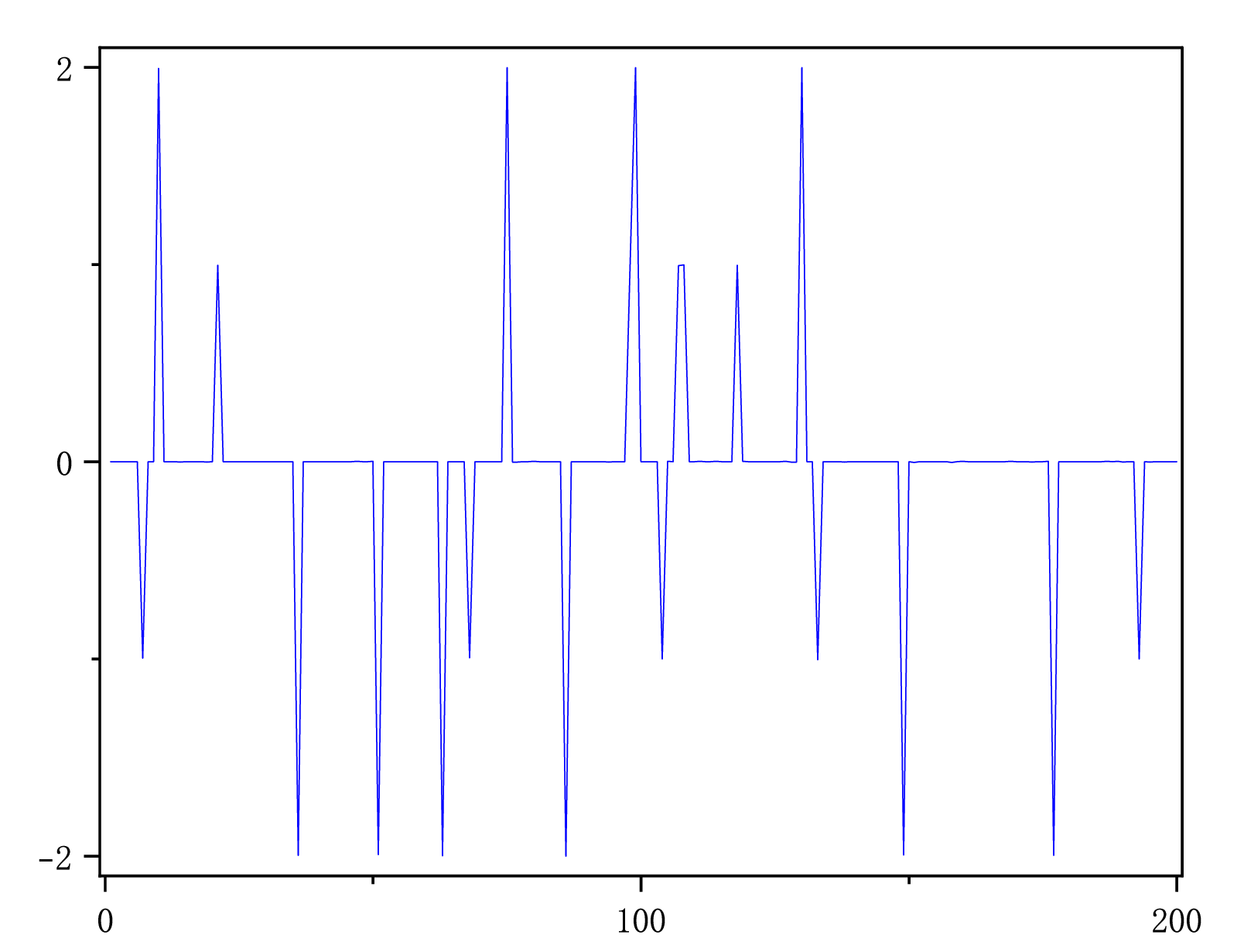}}\\
    \subfigure[Recovery signal by $\ell^{1}$-SVD algorithm ({\rm Rerror}= 0.0028).]{\includegraphics[scale=0.26]{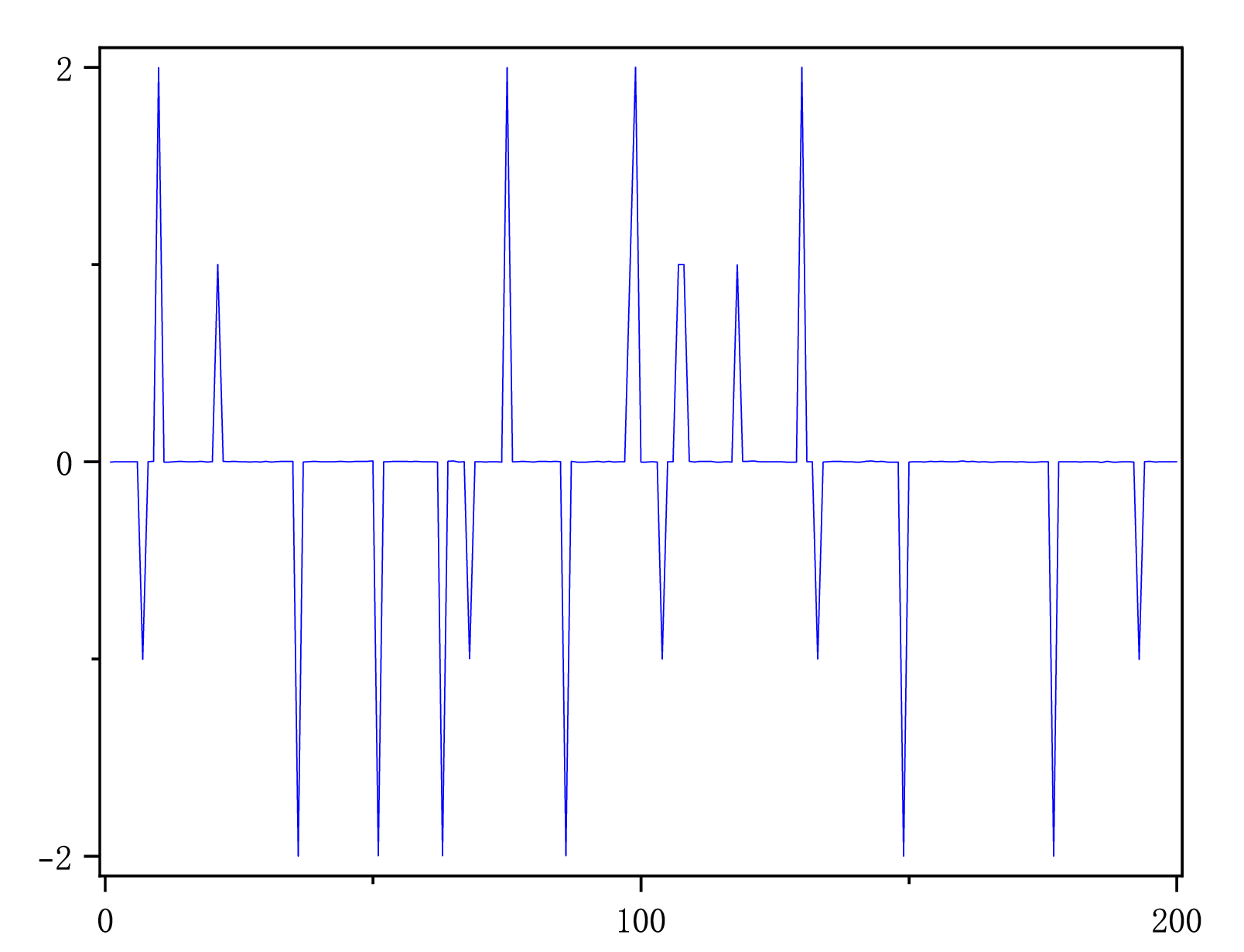}}
    \subfigure[Recovery signal by $\ell^{1/2}$-SVD algorithm ({\rm Rerror}= 0.0026).]{\includegraphics[scale=0.26]{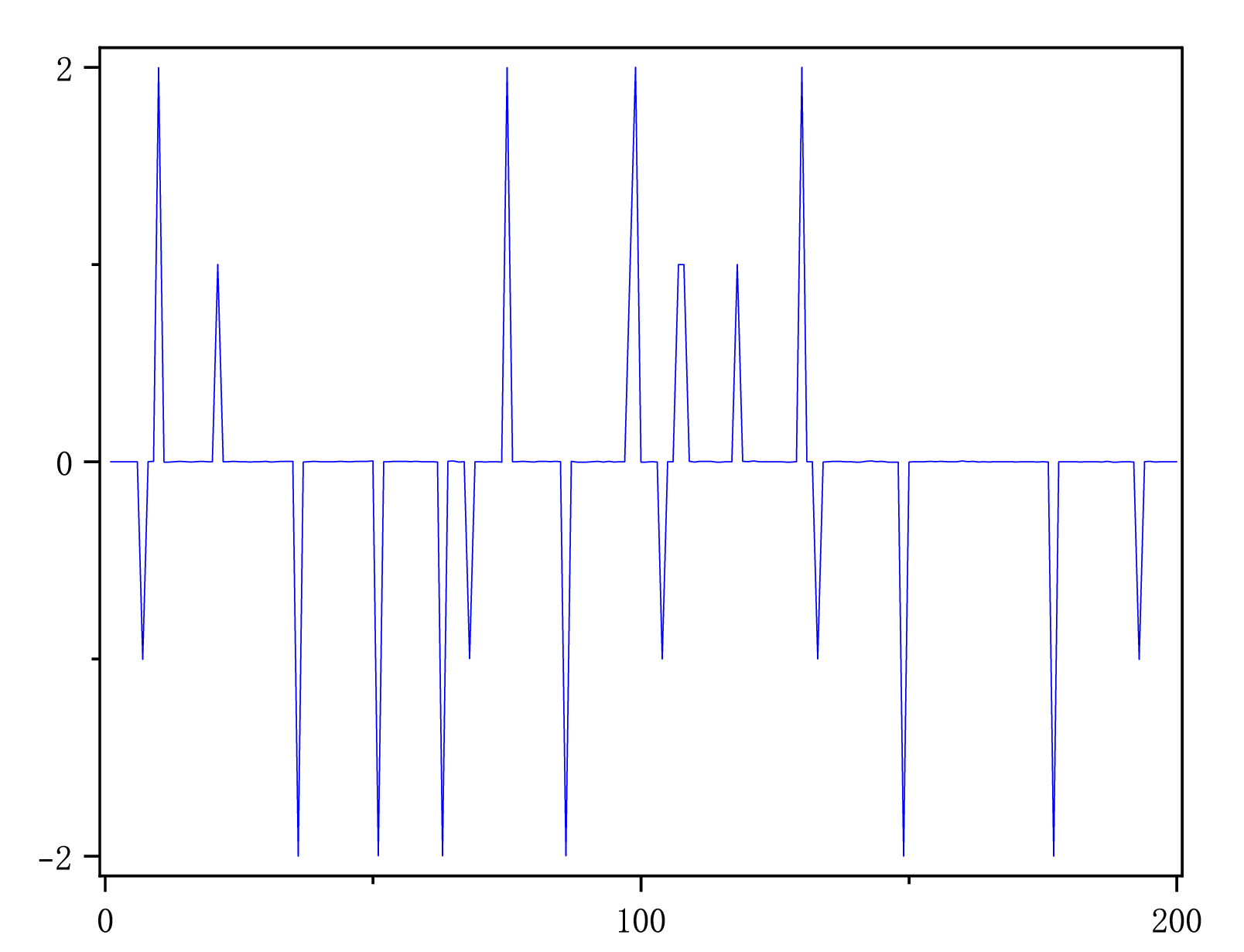}}
    \caption{(a) is the true signal, (b)-(f) are the recovery effect of the ISTA, FISTA, PG, $\ell^{1}$-SVD and $\ell^{1/2}$-SVD algorithms, respectively.}
    \label{figure:2}
\end{figure}

\par We utilize different algorithms to restore signals with different sizes of $K_{m\times n}$ $(where\ m=n)$. In our experiment, we set $s=0.1m$ for the signal and measure the Rerror and running time. For the iterative algorithms, the regularization parameter $\alpha$ is selected based on Morozov's discrepancy principle. For SVD algorithms, we choose $\alpha=\mathcal{O}\left(\delta\right)$. The stopping criterion is that the Rerror between two neighboring iterations is less than 1e-5, or the maximum number of iterations reaches 2000. Fig.\ \ref{figure:1}(a) shows the true signal, and Fig.\ \ref{figure:1}(b) shows the observed data with $80$dB noise. Fig.\ \ref{figure:2}(b)-(f) presents the recovery results from the ISTA, FISTA, PG, and SVD algorithms, respectively. The results demonstrate that the recovery performance of the SVD algorithms is competitive with that of the iterative algorithms

\begin{table}[H]
\caption{The Rerror and running time of different algorithms with different size of $K$ when $\sigma=80$dB.}
\label{table:1}
\centering
\tabcolsep=0.35cm
\tabcolsep=0.01\linewidth
\scalebox{0.9}{
\begin{tabular}{c|cccccccccc}
\hline
Size of $K$ & \multicolumn{2}{c}{ISTA} & \multicolumn{2}{c}{FISTA} & \multicolumn{2}{c}{PG} & \multicolumn{2}{c}{$\ell^{1}$-SVD} & \multicolumn{2}{c}{$\ell^{1/2}$-SVD} \\
    & Rerror & Time(s) & Rerror & Time(s) & Rerror & Time(s) & Rerror & Time(s) & Rerror & Time(s) \\
\hline
200$\times$200  & 0.0078 & 17.55 & 0.0064 & 2.78 & 0.0029 & 0.08 & 0.0028 & \textbf{0.01} & \textbf{0.0026} & \textbf{0.01} \\
400$\times$400  & 0.0041 & 58.34 & 0.0033 & 10.06 & \textbf{0.0016} & 0.33 & 0.0056 & 0.06 & 0.0047 & \textbf{0.04} \\
800$\times$800  & 0.0023 & 554.87 & 0.0019 & 96.35 & \textbf{0.0009} & 1.25 & 0.0026 & 0.19 & 0.0022 & \textbf{0.16} \\
1600$\times$1600  & 0.0010 & 2953.32 & 0.0009 & 597.24 & \textbf{0.0006} & 2.86 & 0.0076 & 1.36 & 0.0054 & \textbf{1.25} \\
3200$\times$3200  & -- & Over 5h  & 0.0006 & 3524.16 & \textbf{0.0006} & 15.31 & 0.0079 & 11.57  & 0.0078 & \textbf{11.48} \\
6400$\times$6400  & -- & Over 5h  & -- & Over 5h & \textbf{0.0009} & 97.56 & 0.0074 & \textbf{84.38}  & 0.0061 & 84.82 \\
\hline
\end{tabular}}
\end{table}

\par Table \ref{table:1} presents the Rerror and the running times of different algorithms across different sizes of $K$ at a noise level of $80$dB. We denote cases where the running time exceeds 5 hours with \verb+"+Over 5h\verb+"+ to indicate that the respective algorithm is not competitive. From Table \ref{table:1}, it is evident that SVD algorithms exhibit competitive performance relative to the other three algorithms in terms of recovery effectiveness as the size of $K$ increases. However, regarding running time, ISTA, FISTA, and PG require significant time for iterations, with ISTA being the slowest, followed by FISTA, which runs slower than PG. In contrast, SVD algorithms, regardless of the regularization, save considerable time during iterations, primarily spending time on the singular value decomposition process, making them faster than the other methods.

\par  For a fixed matrix $K$, the traditional iterative algorithms such as ISTA, FISTA, and PG algorithms require the computation of \eqref{equ1.3} at each iteration. The computational burden of the SVD algorithm primarily stems from the singular value decomposition of the matrix $K$. However, the SVD algorithm only necessitates the computation of singular value decomposition once. Once a singular value system is established, a regularized solution can be obtained using \eqref{equ1.6} or \eqref{equ1.8}. Thus, the SVD algorithm is considerably faster than the iterative algorithms.

\begin{figure}[htbp]
    \centering
    \includegraphics[scale=0.265]{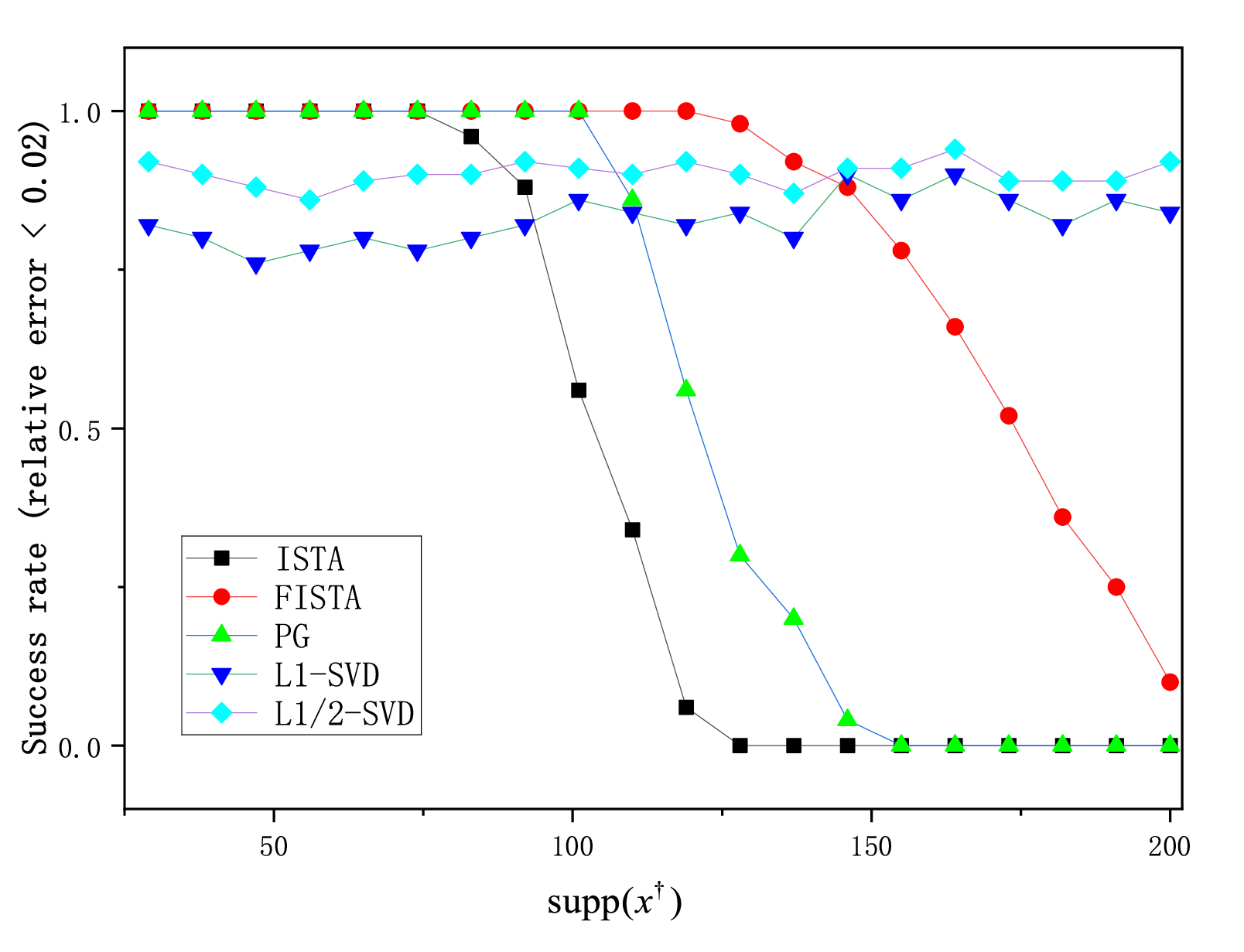}
    \caption{The recovery success probability of ISTA, FISTA, PG and SVD algorithms for data with 80dB nosie, where we keep the size of $K$ fixed and changes ${\rm supp}(x^{\dag})$.}
    \label{figure:3}
\end{figure}

\par Next, we evaluate the probability of successful recovery across different algorithms as the support set ${\rm supp}(x^{\dag})$ increases. With the noise level maintained at  $80$dB, we compare the probability of successful recovery. We select 100 random Gaussian matrices with the same size as a single group and employ different algorithms to evaluate the probability of successful recovery. The results, exhibited in Fig.\ \ref{figure:3}, indicate that as ${\rm supp}(x^{\dag})$ increases, the probability of successful recovery for the ISTA and PG algorithms decreases rapidly, while the decline for the FISTA algorithm is more gradual. In contrast, the probability of successful recovery for the $\ell^{1}$-SVD algorithm ranges between 75\% and 85\%, and for the $\ell^{1/2}$-SVD algorithm, it varies from 85\% to 95\%. These findings suggest that the SVD algorithms for different regularizations demonstrate a significant advantage regarding recovery success probability.

\subsection{Image deblurring}\label{sec4.2}

In the second example, we explore an ill-conditioned image deblurring problem, which pertains to removing blurring artifacts from images caused by factors such as defocus aberration or motion blur. The blur is typically modeled using a Fredholm integral equation of the first kind
\begin{equation*}
   \int_{a}^{b}\mathcal{K}(s,t)f(t)dx=g(s),
\end{equation*}
where $\mathcal{K}(s,t)$ denotes the kernel function, $f(t)$ represents the true image, and $g(s)$ is the observed image. We employ MATLAB regularization tools (\cite{H07}) by calling $[K,y,x^{\dag}]={\rm blur}(n,band,\tau)$, where $K$ is derived from the sparse representation matrix in MATLAB, transformed into a normal representation matrix suitable for singular value decomposition. In this case, the Gaussian point-spread function serves as the kernel function
\begin{equation*}
   \mathcal{K}(s,t)=\frac{1}{\pi\tau^{2}}{\rm exp}\left\{-\frac{s^{2}+t^{2}}{2\tau^{2}}\right\}.
\end{equation*}

\par The symmetric $n^{2}\times n^{2}$ Toeplitz matrix $K$ is expressed as $K=(2\pi\tau^{2})^{-1}T\otimes T$, where $T$ is an $n\times n$ symmetric banded Toeplitz matric. The first row of $T$ is generated by calling
\begin{equation*}
    z=[{\rm exp}\left(-([0:band-1].^{\wedge} 2)\right);{\rm zeros}(1:n^{\wedge} 2-band))].
\end{equation*} 
The parameter $\tau$ controls the shape of the Gaussian point spread function and consequently influences the degree of smoothing (the larger value $\tau$ corresponding, the larger condition number; this implies the deblur problems are more ill-posed). Table \ref{table:2} illustrates the relationship between $\tau$ and ${\rm cond}(K)$, along with the Rerror and the running time of the PG and SVD algorithms for different $\tau$ values, where ${\rm cond}(K)$ represents the condition number of $K$ for blurring $64\times64$ images. In this scenario, ISTA and FISTA require excessive computational time and are therefore not included in the comparison.

\begin{table}[H]
\caption{The Rerror and the running time of the PG and SVD algorithms with with different values of $\tau$.}
\label{table:2}
\centering
\tabcolsep=0.35cm
\tabcolsep=0.01\linewidth
\begin{tabular}{c|cccccccccc}
\hline
$\tau$  & \multicolumn{2}{c}{0.6} & \multicolumn{2}{c}{0.7} & \multicolumn{2}{c}{0.8} & \multicolumn{2}{c}{0.9} \\
\hline
${\rm cond}(K)$  & \multicolumn{2}{c}{8.7327} & \multicolumn{2}{c}{31.328} & \multicolumn{2}{c}{137.11} & \multicolumn{2}{c}{729.84} \\
\hline
        & Rerror & Time(s) & Rerror & Time(s) & Rerror & Time(s) & Rerror & Time(s)  \\
\hline
PG  & \textbf{0.0020} & \textbf{17.72} & \textbf{0.0069} & 88.97 & 0.0375 & 143.00 & \textbf{0.0805} & 143.26 \\ 
$\ell^{1}$-SVD  & 0.0033 & 31.47 & 0.0086 & \textbf{30.62} & 0.0439 & 32.05 & 0.0971 & \textbf{29.99} \\
$\ell^{1/2}$-SVD & 0.0032 & 31.79 & 0.0080 & 30.79 & \textbf{0.0332} & \textbf{30.25} & 0.0816 & 31.17 \\
\hline
\end{tabular}
\end{table}

\begin{figure}[htbp]
    \centering
    \subfigure[True image.]{\includegraphics[scale=0.26]{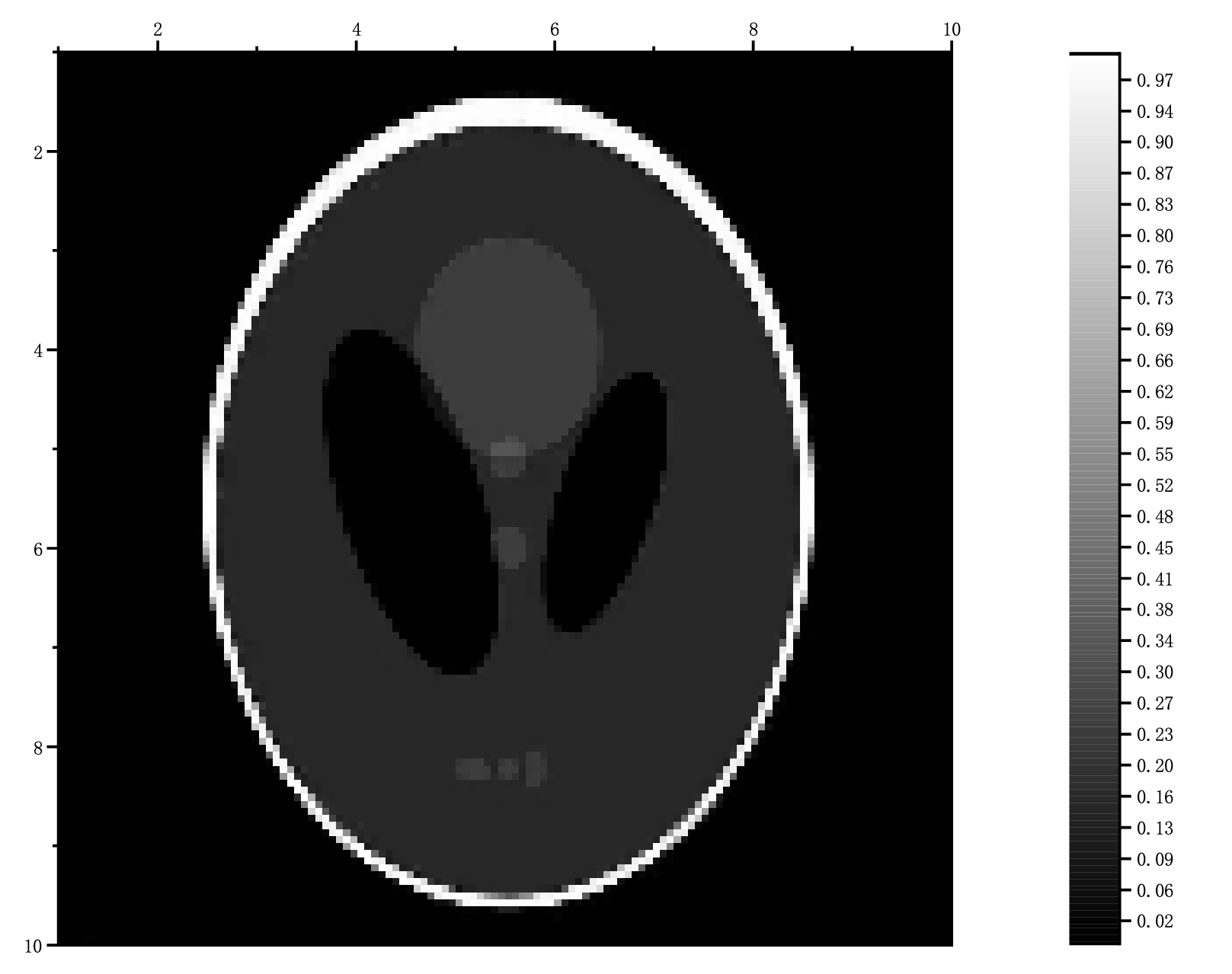}}
    \subfigure[Blurred and noisy image ($\sigma=80$dB.)]{\includegraphics[scale=0.26]{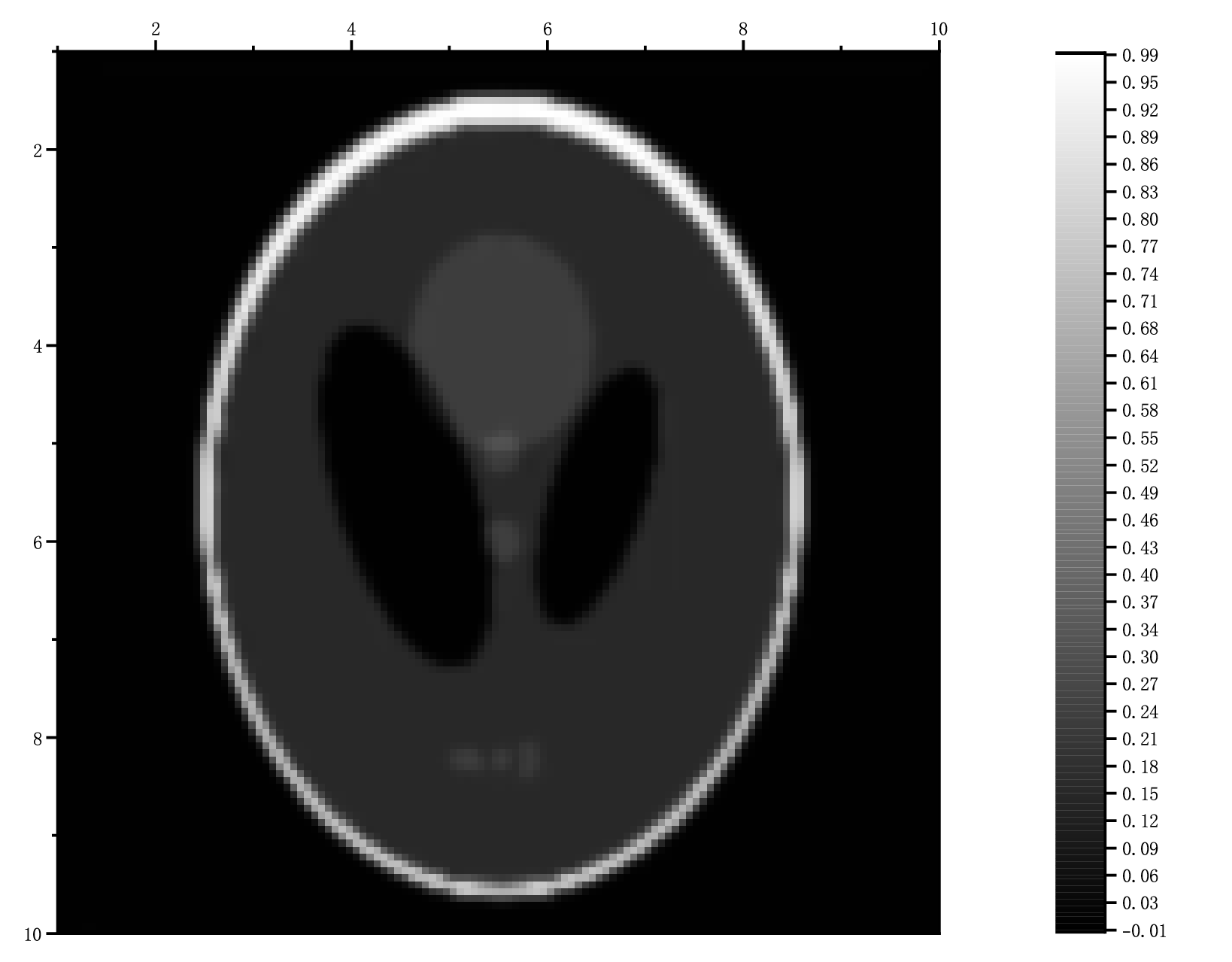}}\\
    \subfigure[Deblurring by PG ({\rm Rerror}=0.0071).]{\includegraphics[scale=0.26]{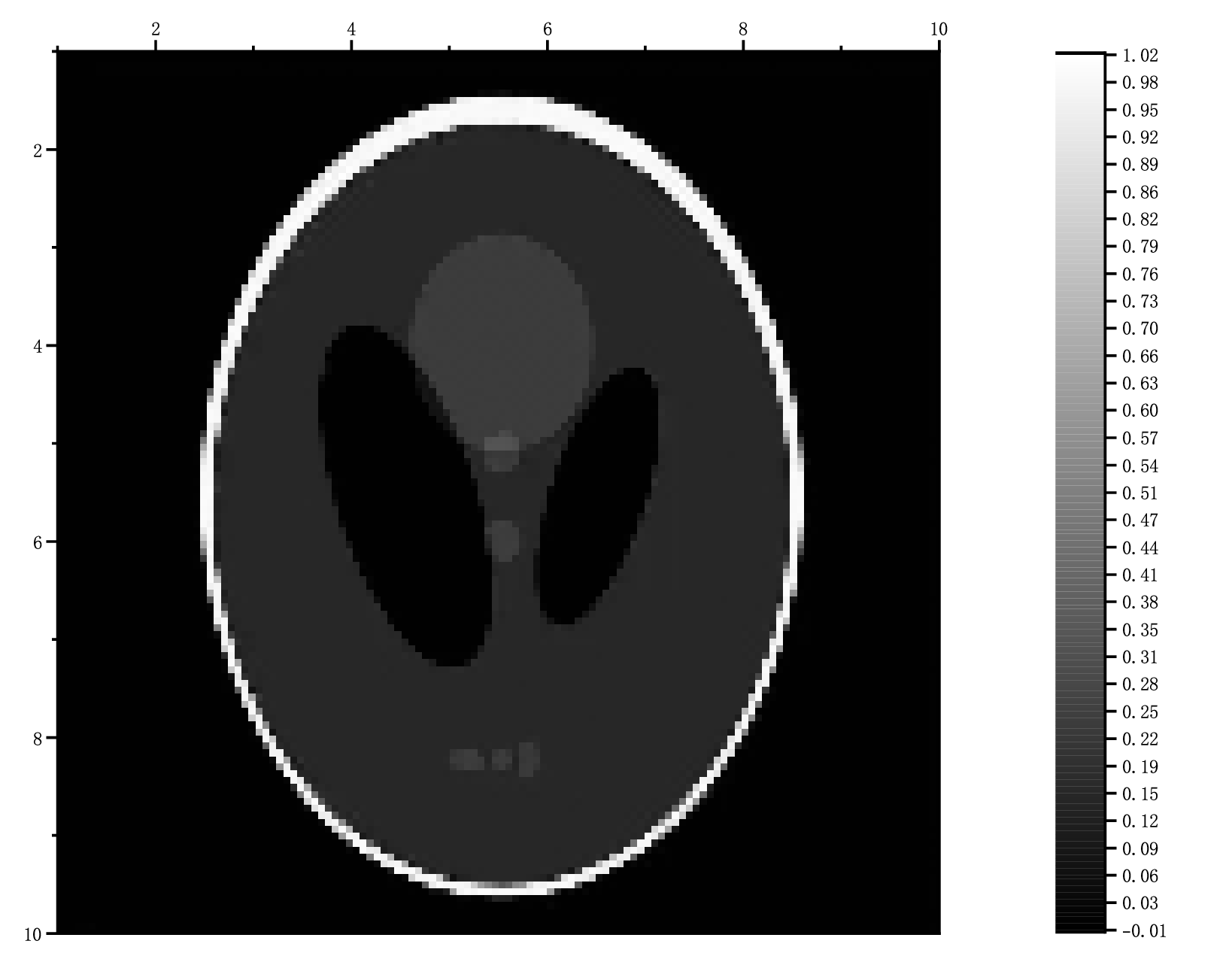}}
    \subfigure[Deblurring by $\ell^{1}$-SVD ({\rm Rerror}=0.0080).]{\includegraphics[scale=0.26]{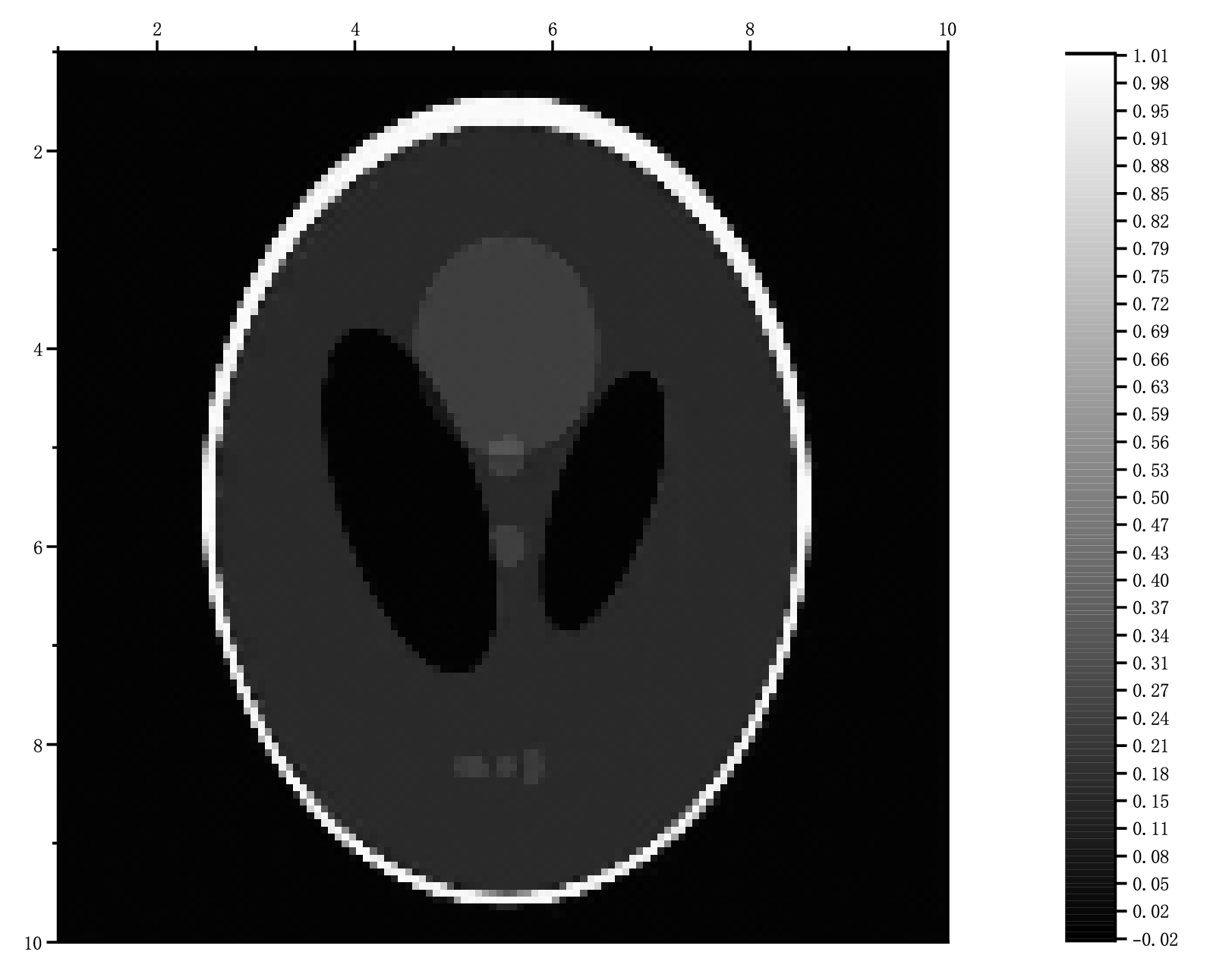}}\\
    \subfigure[Deblurring by $\ell^{1/2}$-SVD ({\rm Rerror}=0.0076).]{\includegraphics[scale=0.26]{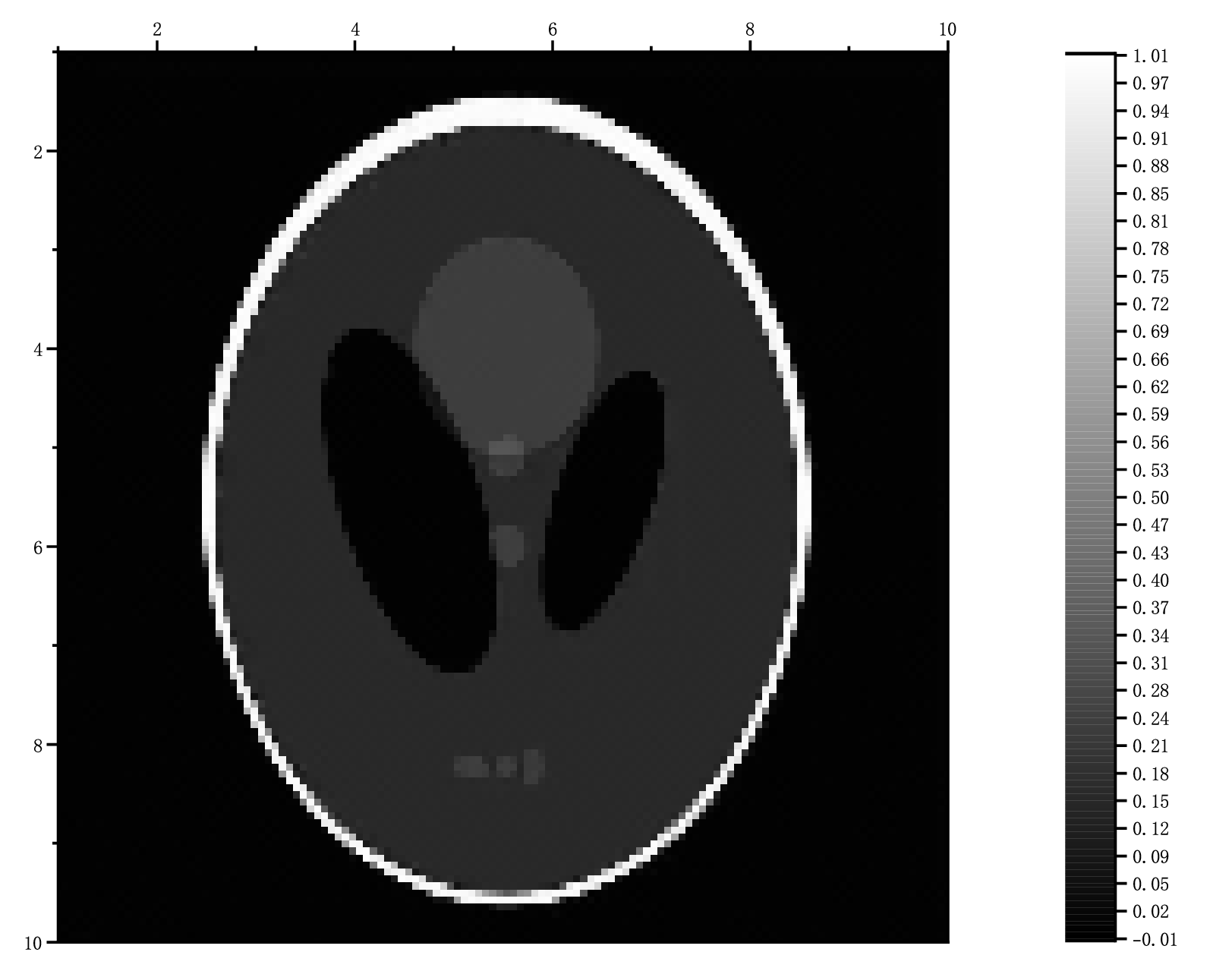}}
    \caption{(a) is the original image and (b) is the blurred image which pixel value is 128$\times$128 with $80$dB noise, (c)-(e) are the recovery effect of the PG, $\ell^{1}$-SVD and $\ell^{1/2}$-SVD algorithms, respectively.}
    \label{figure:4}
\end{figure}

\par We select the parameters as follows $band={\rm floor}(n/4)$, $\tau=0.7$ and add noise $\delta$ to the exact data $y^{\dag}$ by calling $y^{\delta}={\rm awgn}(Kx^{\dag},\sigma)$, where $\sigma=80$dB. For iterative algorithms, the regularization parameter $\alpha$ is determined using Morozov's discrepancy principle. For SVD algorithms, we set $\alpha=\mathcal{O}\left(\delta\right)$. The stopping criterion is that the Rerror between two neighboring iterations is less than 1e-5, or the maximum number of iterations reaches 2000. Fig.\ \ref{figure:4}(a) is the true image, and Fig.\ \ref{figure:4}(b) is the blurred image with $80$dB noise, while Fig.\ \ref{figure:4}(c)-(e) illustrate the recovery outcomes of the PG and SVD algorithms for the case, respectively, where the size of $K$ is $128\times128$ as shown in Fig.\ \ref{figure:4}. The results indicate that the recovery outcomes of the SVD algorithms are competitive with those of the other iterative algorithms.

\begin{table}[H]
\caption{The Rerror and the running time of different algorithms for deblurring with different sizes of $K$ when $\sigma=80$dB.}
\label{table:3}
\centering
\tabcolsep=0.35cm
\tabcolsep=0.01\linewidth
\scalebox{0.88}{
\begin{tabular}{c|cccccccccc}
\hline
The size of  & \multicolumn{2}{c}{ISTA} & \multicolumn{2}{c}{FISTA} & \multicolumn{2}{c}{PG} & \multicolumn{2}{c}{$\ell^{1}$-SVD} & \multicolumn{2}{c}{$\ell^{1/2}$-SVD} \\
matrix $K$  & Rerror & Time(s) & Rerror & Time(s) & Rerror & Time(s) & Rerror & Time(s) & Rerror & Time(s) \\
\hline
$8^{2}\times8^{2}$  & 0.0066 & 0.69 & \textbf{0.0054} & 0.17 & 0.0936 & \textbf{0.01} & 0.0102 & \textbf{0.01} & 0.0094 & \textbf{0.01} \\
$16^{2}\times16^{2}$  & \textbf{0.0067} & 16.65 & 0.0089 & 1.69 & 0.0125 & 0.09 & 0.0090 & 0.03 & 0.0069 & \textbf{0.02} \\
$32^{2}\times32^{2}$  & 0.0061 & 1135.21 & \textbf{0.0046} & 159.61 & 0.0061 & 6.07 & 0.0097 & \textbf{0.41} & 0.0081 & \textbf{0.41} \\
$64^{2}\times64^{2}$  & -- & Over 5h & \textbf{0.0039} & 9525.09 & 0.0072 & 82.66 & 0.0093 & \textbf{30.41} & 0.0083 & 30.45 \\
$128^{2}\times128^{2}$  & -- & Over 5h  & -- & Over 5h & \textbf{0.0071} & 2201.68 & 0.0080 & \textbf{1446.50}  & 0.0076 & 1521.95\\
\hline
\end{tabular}}
\end{table}

\par Table \ref{table:3} presents the Rerror and running time of different algorithms across different sizes of $K$ at a noise level of $80$dB. We denote cases where the running time exceeds 5 hours with \verb+"+Over 5h\verb+"+, indicating that the respective algorithm is not considered competitive. From Table \ref{table:3}, it is evident that the performance of the SVD algorithms remains competitive with the other three algorithms concerning deblurring effectiveness as the size of $K$ increases. Regarding running time, ISTA takes more time than both the FISTA and PG algorithms, with FISTA running longer than PG. In contrast, the SVD algorithms for different regularization methods outperform the others in speed, as most running time is spent on processing singular value decomposition. When the operator $K$ for the linear inverse problems is known, the results of singular value decomposition can be treated as a known condition.

\begin{figure}[htbp]
    \centering
    \includegraphics[scale=0.27]{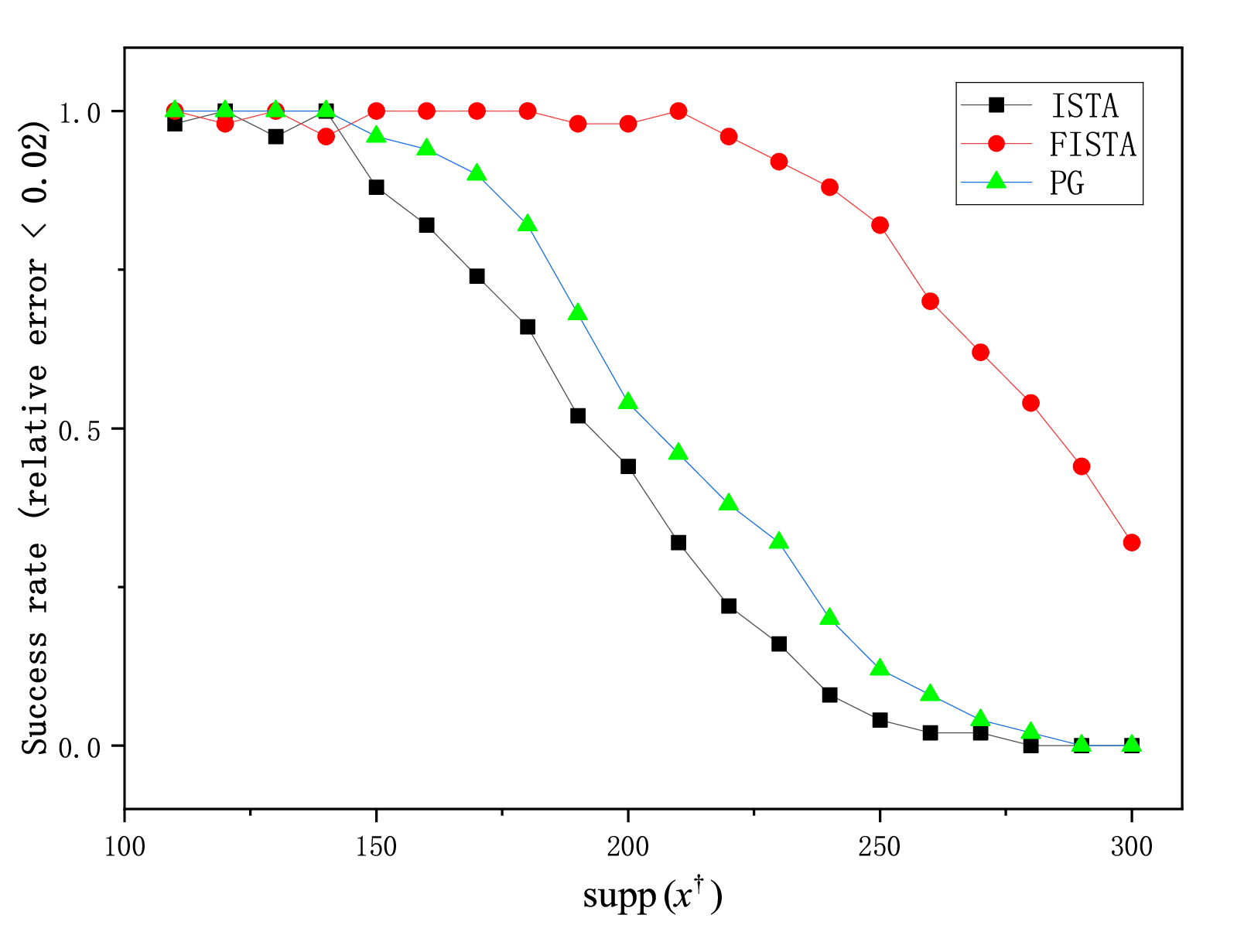}
    \includegraphics[scale=0.27]{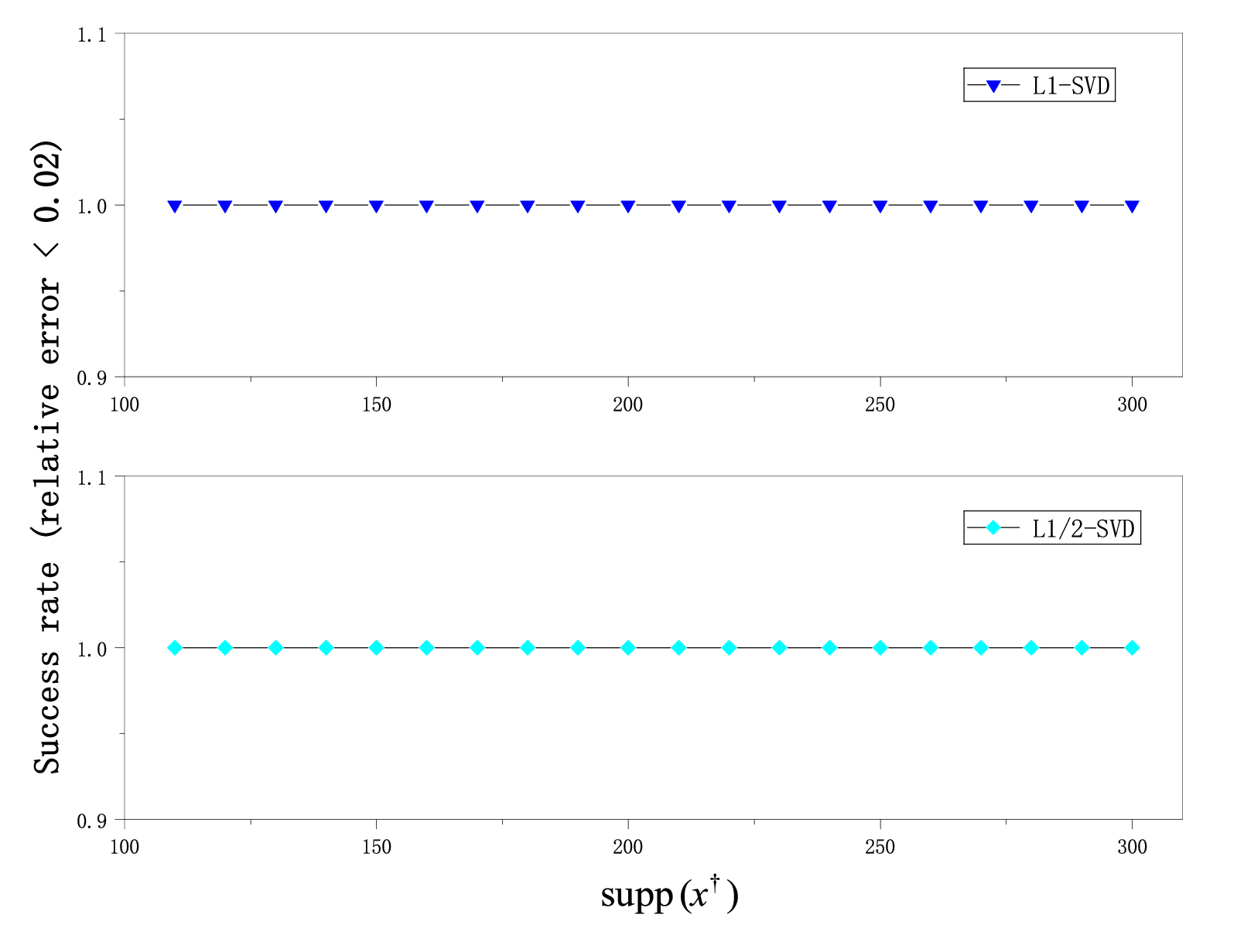}
    \caption{The success rates of ISTA, FISTA, PG and SVD algorithms for data with 80dB nosie, where we keep the size of $K$ fixed and changes ${\rm supp}(x^{\dag})$.}
    \label{figure:5}
\end{figure}

\par What remains is to evaluate the probability of successful recovery for different algorithms as complexity increases. We maintain consistent settings with a noise level of $80$dB and compare the success recovery probabilities of image deblurring across different algorithms. For each ${\rm supp}(x^{\dag})$, we select 100 images that share the same ${\rm supp}(x^{\dag})$ as a group. By altering the value of ${\rm supp}(x^{\dag})$ and employing different algorithms, we evaluate the probability of successful recovery. The results are illustrated in Fig.\ \ref{figure:5}. It is evident that as the support set increases, the probability of successful recovery for both the ISTA and PG algorithms declines rapidly, while the probability of successful recovery for the FISTA algorithm decreases more gradually. In contrast, the SVD algorithms consistently achieve a 100\% success recovery rate. This indicates a significant advantage of the SVD algorithms over different regularization methods regarding success recovery probability.

\section{Conclusion}

\par This paper investigates challenges associated with compressive sensing and image deblurring within the context of sparse inverse problems. By integrating the sparsity regularization with $\ell^{1}$ and $\ell^{1/2}$ penalties using the singular value decomposition method when the operator $K$ is diagonal within a specific orthogonal basis, we propose two new regularization operators that utilize nonlinear soft thresholding and half thresholding functions and provide a theoretical foundation to demonstrate these operators' effectiveness and regularization properties for the general linear compact operator $K$. Our experiments assess the Rerror, running time, and probability of successful recovery. The results indicate that the performance of the SVD algorithms is competitive with the ISTA, FISTA, and PG algorithms concerning Rerror. In contrast, the SVD algorithms outperform the other three iterative algorithms regarding running time and probability of successful recovery.

\section{Conflict of interest statement}
There are no conflicts of interest regarding this submission, and all authors have approved the manuscript for publication.

\end{document}